\documentclass[12pt]{article}
\usepackage{a4wide}
\usepackage{amsmath}
\usepackage{amsthm}
\usepackage{amsfonts}
\usepackage{amssymb}
\usepackage{cite}
\usepackage{epsfig}
\newtheorem{theorem}{Theorem}
\newtheorem{lemma}[theorem]{Lemma}
\newtheorem{conjecture}{Conjecture}
\def\CC{{\cal C}}
\def\NN{{\mathbb N}}
\def\ZZ{{\mathbb Z}}
\def\RR{{\cal R}}
\def\mod{\;{\rm mod}\;}
\def\parnt{\mathbf{parnt}}
\def\succ{\mathbf{succ}}

\def\FOlocal{{\rm FO}_1^{\rm local}}
\begin{document}
\title{First order limits of sparse graphs:\\ Plane trees and path-width}
\author{Jakub Gajarsk\'y\thanks{Faculty of Informatics, Masaryk University, Botanick\'a~68a, 602~00~Brno, Czech Republic. E-mail: {\tt \{xgajar,hlineny,obdrzalek,ordyniak\}@fi.muni.cz}. JG, PH and JO have been supported by project 14-03501S of the Czech Science Foundation. SO has been supported by the European Social Fund and the state budget of the Czech Republic under project CZ.1.07/2.3.00/30.0009 (POSTDOC I).}\and
\newcounter{lth}
\setcounter{lth}{1}
        Petr Hlin\v en\'y$^\fnsymbol{lth}$\and
        Tom\'a\v s Kaiser\thanks{Department of Mathematics, Institute for Theoretical Computer Science (CE-ITI), and European Centre of Excellence NTIS (New Technologies for the Information Society), University of West Bohemia, Univerzitn\'{\i} 8, 306~14 Pilsen, Czech Republic. Email: {\tt kaisert@kma.zcu.cz}. Supported by project GA14-19503S of the Czech Science Foundation.}\and
        Daniel Kr\'al'\thanks{Mathematics Institute, DIMAP and Department of Computer Science, University of Warwick, Coventry CV4 7AL, UK. E-mail: {\tt d.kral@warwick.ac.uk}. This author's work was supported by the European Research Council under the European Union's Seventh Framework Programme (FP7/2007-2013)/ERC grant agreement no.~259385 and by the Engineering and Physical Sciences Research Council Standard Grant number EP/M025365/1.}\and
        Martin Kupec\thanks{Computer Science Institute, Faculty of Mathematics and Physics, Charles University, Malo\-stransk\'e n\'am\v est\'i~25, 118~00, Prague, Czech Republic. E-mail: {\tt magon@iuuk.mff.cuni.cz}.}\and
        Jan Obdr\v{z}\'alek$^\fnsymbol{lth}$\and
        Sebastian Ordyniak$^\fnsymbol{lth}$\and
	Vojt\v{e}ch T\r{u}ma\thanks{Department of Applied Mathematics, Faculty of Mathematics and Physics, Charles University, Malostransk\'e n\'am\v est\'i~25, 118~00, Prague, Czech Republic. E-mail: {\tt voyta@kam.mff.cuni.cz}.}
	}
\date{}
\maketitle
\begin{abstract}
Ne\v set\v ril and Ossona de Mendez introduced the notion of first order convergence as
an attempt to unify the notions of convergence for sparse and dense graphs.
It is known that there exist first order convergent sequences of graphs with no limit modeling (an analytic representation of the limit).
On the positive side, every first order convergent sequence of trees or
graphs with no long path (graphs with bounded tree-depth) has a limit modeling.
We strengthen these results by showing that
every first order convergent sequence of plane trees (trees with embeddings in the plane) and
every first order convergent sequence of graphs with bounded path-width has a limit modeling.
\end{abstract}

\section{Introduction}
\label{sect-intro}

The theory of combinatorial limits has quickly become an important area of combinatorics.
The most developed is the theory of graph limits,
which is a subject of a recent monograph by Lov\'asz~\cite{bib-lovasz-book}.
The graph convergence evolved to a large extent differently and independently for dense and sparse graphs.
The case of dense graphs was developed in the series of papers
by Borgs, Chayes, Lov\'asz, S\'os, Szegedy and Vesztergombi~\cite{bib-borgs08+,bib-borgs+,bib-borgs06+,bib-lovasz06+,bib-lovasz10+}
and is considered to be well-understood.
In the case of sparse graphs (such as those with bounded maximum degree),
the most used notion of convergence known as the Benjamini-Schramm convergence,
which was studied e.g.~in~\cite{bib-aldous07+,bib-benjamini01+,bib-elek07},
comes with substantial disadvantages.
Several alternative notions were proposed~\cite{bib-bollobas11+,bib-borgs13++,bib-hatami+}, however,
each of them also comes with certain drawbacks.

As an attempt to unify the existing notions of convergence for dense and sparse graphs,
Ne\v set\v ril and Ossona de Mendez~\cite{bib-pom,bib-pom-unified} proposed
a notion of convergence based on first order properties of graphs,
the {\em first order convergence} (a formal definition is given in Section~\ref{sect-fo}).
This notion applies to all relational structures, and
it implies the standard notion of convergence in the case of dense graphs and
the Benjamini-Schramm convergence in the case of sparse graphs.
A first order convergent sequence of graphs can be associated with an analytic representation,
known as a {\em limit modeling}.
Unfortunately, not all first order convergent sequences of graphs do have a limit modeling~\cite{bib-pom-unified},
e.g., the sequence of Erd\H os-R\'enyi random graphs is first order convergent with probability one
but it has no limit modeling.

The existence of a limit modeling of a first order convergent sequence
is one of central problems related to first order convergence.
Ne\v set\v ril and Ossona de Mendez~\cite{bib-pom-unified} conjectured that
every first order convergent sequence of sparse graphs has a limit modeling:
\begin{conjecture}
\label{conj}
Let $\CC$ be a nowhere-dense class of graphs.
Every first order convergent sequence of graphs from $\CC$ has a limit modeling.
\end{conjecture}
\noindent Recall that nowhere-dense classes of graphs are classes of graphs~\cite{bib-nom2012}
which include all minor closed classes of graphs (in particular, trees, planar graphs, etc.) and
some more general classes of sparse graphs.
However, only little is known towards proving Conjecture~\ref{conj}.
Ne\v set\v ril and Ossona de Mendez~\cite{bib-pom-unified} showed that
every first order convergent sequence of trees of bounded depth has a limit modeling, and
they used this result to show that
every first order convergent sequence of graphs with bounded tree-depth has a limit modeling.
Three of the authors (DK, MK and VT) extended this result and showed that
every first order convergent sequence of trees has a limit modeling.
This is also implied by a more general result of Ne\v set\v ril and Ossona de Mendez~\cite{bib-pom-trees},
who developed a framework for building limit modelings based
on residual and non-dispersive first order convergent sequences.

In this paper, we make another step towards a proof of Conjecture~\ref{conj}.
We show that every first order convergent sequence of trees embedded in the plane has a limit modeling (Theorem~\ref{thm-plane}) and
that every first order convergent sequence of graphs with bounded path-width has a limit modeling (Theorem~\ref{thm-pw}).
While the first result can be viewed as a small extension of the result on the existence of limit modelings of first order convergent sequences of trees,
it turned out that embedding the trees in the plane, which essentially corresponds to fixing the cyclic order among the neighbors of each vertex,
gave us enough power to prove the (more important) result
on the existence of limit modelings of first order convergent sequences of graphs with bounded path-width.
Note that the class of graphs of bounded path-width is significantly richer than
the class of trees, which do not have cycles at all, or the classes of graphs with bounded tree-depth, which do not have long paths.
In a certain sense, this is the first class of graphs with rich internal structure for which Conjecture~\ref{conj} is proven.

The proof of Theorem~\ref{thm-plane} on the existence of limit modelings of first order convergent sequences of plane trees
consists of two steps: a decomposition step described in Section~\ref{sect-decompose} and
a composition step described in Section~\ref{sect-compose}.
The decomposition step aims at analyzing first order properties of the graphs in the sequence and
describing them through quantities that we refer to as Stone measure and discrete Stone measure.
The composition step then uses these quantities to build a limit modeling of the sequence.
The decomposition step follows the line of our original proof of the existence of limit modelings of first order convergent sequences of trees.
However, we decided to replace the composition step of our original proof
with the arguments from the analogous part of the proof given in~\cite{bib-pom-trees},
which we have found elegant and simpler than the composition step of our original proof using methods from~\cite{bib-hatami+}.
We then employ first order interpretation schemes to encode graphs with bounded path-width by plane trees,
which allows us to prove our result on first order convergent sequences of graphs with bounded path-width (Theorem~\ref{thm-pw}).
We also note that
the modelings constructed in Theorems~\ref{thm-plane} and~\ref{thm-pw} satisfy the strong finitary mass transport principle,
which, vaguely speaking, forbids the existence of a small and a large subset of vertices with a matching between them.

\section{Notation}
\label{sect-notation}

We mostly follow the standard graph theory terminology and the standard model theory terminology as
it can be found e.g.~in~\cite{bib-diestel} and in~\cite{bib-ebbinghaus+}, respectively.
Still, we want to specify some less standard notation and
to introduce some non-standard notation related to graphs with bounded path-width.
In what follows, all graphs, trees, etc. are finite unless specified otherwise, and
the {\em order} of a graph is the number of its vertices.
The set of positive integers is denoted by $\NN$ and
the set of non-negative integers by $\NN_0$. %and the set of all integers by $\ZZ$.
The set of integers from $1$ to $k$ (inclusively) is denoted by $[k]$, and $\NN^*$ stands for $\NN\cup\{\infty\}$.
If $x$ is a real number and $z$ is a positive real, $x\mod z$ denotes the unique real $x'\in [0,z)$ such that
$x=x'+kz$ for some $k\in\ZZ$.

\subsection{Path-width and semi-interval graphs}
\label{subsect-pw}

If $G$ is a graph, then a semi-interval representation of $G$
is an assignment of intervals $J_v$ of the form $[k,\ell)$, $k,\ell\in\ZZ$, $k<\ell$, to vertices $v$ of $G$ such that
the intervals $J_v$ and $J_{v'}$ of any two adjacent vertices $v$ and $v'$ intersect (however,
the intervals of non-adjacent vertices may also intersect).
A graph with a fixed semi-interval representation is called a {\em semi-interval graph}.
The intervals of the form $[k,k+1)$, $k\in\ZZ$ are called segments.
If $G$ is a semi-interval graph, then the {\em first segment} is the leftmost segment intersected by an interval assigned to a vertex of $G$ and
the {\em last segment} is the rightmost such segment.

The {\em path-width} of a graph $G$ is the smallest integer $k$ such that
$G$ has a semi-interval representation such that each segment is contained in at most $k+1$ intervals assigned to vertices of $G$.
Note that this definition coincides with the usual definition of the path-width.
In particular, given a semi-interval graph $G$ such that each segment is contained in at most $k+1$ intervals,
one can construct its path-decomposition of width (at most) $k$
by taking a path with vertices corresponding to the segments between the first and the last segment and
assigning each vertex $u$ of the path a bag consisting of the vertices of $G$ whose intervals contain the segment corresponding to $u$.
Likewise, a path-decomposition $G$ of width $p$ naturally yields a semi-interval graph such that
each segment is contained in at most $p+1$ intervals.

\subsection{First order convergence}
\label{sect-fo}

The notion of first order convergence applies to all relational structures (and even further, e.g., to matroids~\cite{bib-kardos+}).
However, we limit our exposition to graphs for simplicity.
The extensions to rooted graphs, vertex-colored graphs, etc., are straightforward.
If $\psi$ is a first order formula with $k$ free variables and $G$ is a (finite) graph,
then the {\em Stone pairing} $\langle \psi,G\rangle$ is the probability that
a uniformly chosen $k$-tuple of vertices of $G$ satisfies $\psi$.
A sequence $(G_n)_{n\in\NN}$ of graphs is {\em first order convergent}
if the limit $\lim\limits_{n\to\infty}\langle \psi,G_n\rangle$ exists for every first order formula $\psi$.
We note that every sequence of graphs has a first order convergent subsequence.

A {\em modeling} $M$ is a (finite or infinite) graph 
with a standard Borel space on its vertex set equipped with a probability measure
such that the set of all $k$-tuples of vertices of $M$ satisfying a formula $\psi$ is measurable
in the product measure for every first order formula $\psi$ with $k$ free variables.
In the analogy to the graph case, the {\em Stone pairing} $\langle \psi,M\rangle$
is the probability that a randomly chosen $k$-tuple of vertices satisfies $\psi$.
If a finite graph is viewed as a modeling with a uniform discrete probability measure on its vertex set,
then the Stone pairings for the graph and the modeling obtained in this way coincide.
A modeling $M$ is a {\em limit modeling} of a first order convergent sequence $(G_n)_{n\in\NN}$
if $$\lim_{n\to\infty}\langle \psi,G_n\rangle=\langle\psi,M\rangle$$
for every first order formula $\psi$.
The definitions of a modeling and a limit modeling readily generalize from the case of graphs
to the case of general relational structures, which include directed graphs, graphs with colored edges, etc. as particular cases.

Every limit modeling $M$ of a first order convergence sequence of graphs satisfies
the so-called {\em finitary mass transport principle} (see~\cite{bib-pom-trees} for further details) that requires that
for any two first order formulas $\psi$ and $\psi'$, each with one free variable, such that
every vertex $v$ satisfying $\psi(v)$ has at least $a$ neighbors satisfying $\psi'$ and
every vertex $v$ satisfying $\psi'(v)$ has at most $b$ neighbors satisfying $\psi$,
it holds that
$$a\langle\psi,M\rangle\le b\langle\psi',M\rangle\;\mbox{.}$$
We are interested in a stronger variant of this principle.
We say that a modeling $M$ satisfies the {\em strong finitary mass transport principle}
if every two measurable subsets $A$ and $B$ of the vertices of $M$ such that
each vertex of $A$ has at least $a$ neighbors in $B$ and
each vertex of $B$ has at most $b$ neighbors in $A$
satisfy that
$$a\mu(A)\le b\mu(B)$$
where $\mu$ is the probability measure of $M$.
Note that the assertion of the finitary mass transport principle requires this inequality to hold only for first order definable subsets of vertices.
The strong finitary mass transport principle is satisfied by any finite graph when viewed as a modeling
but it need not hold for every limit modeling.
The importance of the strong finitary mass transport principle comes from its relation to graphings,
which are limit representations of Benjamini-Schramm convergent sequences of bounded degree graphs:
a limit modeling of a first order convergent sequence of bounded degree graphs is a limit graphing of the sequence if and only if
$M$ satisfies the strong finitary mass transport principle.

\subsection{Hintikka chains}
\label{sect-Hintikka}

Hintikka sentences are maximally expressive sentences with a certain quantifier depth.
We alter the definition to local formulas with a single free variable.
Consider a signature that includes the signature of graphs and is finite except that
it may contain countably many constants, which are labelled by natural numbers.
A first order formula $\psi$ is {\em local}
if each quantifier is restricted to the neighbors of one of the vertices,
i.e., it is of the form $\forall_z x$ or $\exists_z x$ where $x$ is required to be a neighbor of $z$.
For example, the formula $\psi(z)\equiv\exists_z x\;\forall_x y\; y=z$ is true
iff the vertex $z$ has a neighbor of degree exactly one.
The quantifier depth of a local formula is defined in the usual way.
A formula is a {\em $d$-formula} if its quantifier depth is at most $d$ and
it does not involve any constants except for the first $d$ constants;
a $d$-formula with no free variables is referred to as a {\em $d$-sentence}.

The same argument as in the textbook case of first order sentences yields that
there are only finitely many non-equivalent local $d$-formulas with one free variable for every $d$.
Let $\FOlocal$ be a maximal set of non-equivalent local formulas with one free variable,
i.e., a set containing one representative from each equivalence class of local formulas with one free variable.
If $d$ is an integer, the {\em $d$-Hintikka type} of a vertex $v$ of a (not necessarily finite) graph $G$
is the set of all $d$-formulas $\psi\in\FOlocal$ such that $G\models\psi(v)$.
A formula $\psi\in\FOlocal$ with quantifier depth at most $d$ is called a {\em $d$-Hintikka formula}
if there exist a (not necessarily finite) graph $G$ and a vertex $v$ of $G$ such that
$\psi$ is equivalent to the conjunction of the {\em $d$-Hintikka type} of the vertex $v$ of the graph $G$
(note that $\psi$ must actually be equivalent to one of the formulas in the $d$-Hintikka type of $v$).

Fix a $d$-Hintikka formula $\psi$.
Observe that if $v$ is a vertex of a graph $G$ and $v'$ is a vertex of a graph $G'$ such that $G\models\psi(v)$ and $G'\models\psi(v')$
then the $d$-Hintikka types of $v$ and $v'$ are the same. So, we can speak of the $d$-Hintikka type of $\psi$.
A {\em Hintikka chain} is a sequence $(\psi_i)_{i\in\NN}$ such that $\psi_i$ is an $i$-Hintikka formula and
the $i$-Hintikka type of $\psi_i$ contains $\psi_1,\ldots,\psi_{i-1}$.
If $\psi$ is a $d$-Hintikka formula, the set $\{(\psi_i)_{i\in\NN}:\psi_d=\psi\}$ of Hintikka chains is called {\em basic}.
The set of Hintikka chains (formed by Hintikka formulas with a fixed signature) can be equipped
with the topology with the base formed by basic sets; basic sets are clopen in this topology.
This defines a Polish space on the set of Hintikka chains.
In the proof of Lemma~\ref{lm-composition},
we will define a measure on the $\sigma$-algebra of Borel sets of Hintikka chains.
Let us remark that the just defined topological space of Hintikka chains
is homeomorphic to the Stone space of $\FOlocal$ studied in~\cite{bib-pom,bib-pom-unified}.

A graph theory inspired view of Hintikka chains can be the following.
Consider a rooted infinite tree where the root node corresponds to the tautology and
the nodes at depth $d$ one-to-one correspond to $d$-Hintikka formulas.
The parent of the node corresponding to a $d$-Hintikka formula $\psi$ is the unique node $u$ at depth $d-1$ such that
the formula corresponding to $u$ is contained in the $d$-Hintikka type of $\psi$.
Note that the constructed tree is locally finite.
The Hintikka chains one-to-one correspond to infinite paths from the root and
the just defined topology is the most often considered topology on such paths in the infinite rooted trees.

\section{Limits of plane trees}
\label{sect-plane}

A {\em plane tree} is a rooted tree embedded in the plane.
Having in mind our application to graphs with bounded path-width,
we will refer to vertices of plane trees as {\em nodes} (to distinguish
them from the vertices of graphs with bounded path-width, which we consider later).
The signature we use to describe plane trees consists of two binary relational symbols $\parnt$ and $\succ$:
the relation $\parnt$ describes the child-parent relation and
the relation $\succ$ determines a linear order among the children of nodes.
More precisely, if $T$ is a plane tree,
then $\parnt(x,y)$ if $x$ is a child of $y$.
The embedding of $T$ in the plane determines a linear order among the children of each node $u$ and
$\succ(x,y)$ if $x$ and $y$ are children of the same node that are consecutive in this linear order.
Note that $\succ$ determines a linear order on the children of each node
but this order is not first order definable using $\succ$.
Finally, since plane trees are rooted, we can speak about subtrees of their nodes,
i.e.~the {\em subtree} of a node $u$ of a plane tree is the subtree induced by all descendants of $u$.

The main result of this section is the following theorem.

\begin{theorem}
\label{thm-plane}
Every first-order convergent sequence of plane trees has a limit modeling
satisfying the strong finitary mass transport principle.
\end{theorem}

As we have already mentioned, the proof of Theorem~\ref{thm-plane} has two parts.
First, we represent a first order convergent sequence of plane trees through a measure on Hintikka chains.
This will require identifying nodes in the sequence that are in the proximity of a non-zero fraction of all the nodes.
The second step of the proof consists of constructing a tree on a measurable set of nodes that
satisfies the same first order formulas and has the same Stone pairings as the considered sequence of plane trees.
On a high level, one can think of our decomposition step as analogous to comb structure results
given for trees of bounded depth in~\cite[Theorems 27--29]{bib-pom-unified} and
for general trees in~\cite[Theorem 36]{bib-pom-trees}.
However, we do not use comb structure to capture first order properties of plane trees and
we employ another technique, bringing a different view on constructing limits of first order convergent sequences.
We believe that our decomposition technique involves fewer tools from analysis and model theory and is more of combinatorial nature,
which might be useful when extending the results to wider classes of graphs.
On the other hand,
the composition step developed by some of the authors when constructing modelings of first order convergent sequences of trees
was more technical and less elegant than the one used in~\cite{bib-pom-trees}.
So, we adapt ideas from the proof of Lemma 39 in~\cite{bib-pom-trees} to prove Lemma~\ref{lm-composition};
however, it is not possible to use Lemma 39 from~\cite{bib-pom-trees} in our setting directly
since it concerns trees/forests only and
it is designed to be used in conjunction with comb structure results presented in~\cite{bib-pom-trees}.

\subsection{Decomposition}
\label{sect-decompose}

Our first aim is to identify nodes that are in the proximity of a non-zero fraction of other nodes.
A node $u$ of a tree $T$ is {\em $\varepsilon$-major}
if the sum of the sizes of the two largest components of $T\setminus u$ is at most $(1-\varepsilon)|T|$.
Every node in the proximity of a non-zero fraction of other nodes
is also close to a major node, as given in the next lemma.
However, the converse need not be true,
i.e.~the $r$-neighborhoods of major nodes can be small.

\begin{lemma}
\label{lm-major-neighborhood}
Let $T$ be a tree, $\varepsilon$ a positive real and $U$ the set of $\varepsilon$-major nodes of $T$.
For every node $v$ of $T\setminus U$ and every integer $r$,
the number of nodes at distance at most $r$ from $v$ in $T\setminus U$ is at most $(2^{r+1}+1)\varepsilon|T|$.
\end{lemma}

\begin{proof}
Fix the node $v$ of $T\setminus U$. We will color some nodes in the $r$-neighborhood of $v$ as green.
We start with coloring the node $v$ green and proceed as follows as long as we can.
If $w$ is a green node at distance at most $r-1$ from $v$,
we color the neighbors of $w$ in the two largest components of $T\setminus w$ green.
In this way, at most $2^{r+1}+1$ nodes are colored green (each node at distance $\ell$ from $v$
is adjacent to at most two green nodes at distance $\ell+1$).
Next, we recolor all green nodes that are $\varepsilon$-major to red.
We will refer to the nodes that are not red or green as black nodes.

A node can be at distance at most $r$ from $v$ in $T\setminus U$ only if it is joined to $v$ by a path consisting of green and black nodes only.
If $w$ is a green node, then $w$ is not $\varepsilon$-major and
the sum of the orders of the components of $T\setminus w$ such that the neighbor of $w$ in the component is black
is at most $\varepsilon|T|-1$ (the neighbors of $w$ in the two largest components are green or red).
Hence, the number of nodes reachable from $v$ through a path consisting of green and black nodes only such that $w$
is the last green node on the path is at most $\varepsilon|T|$ (here, we also count the node $w$).
Since there are at most $2^{r+1}+1$ green nodes, we conclude that the $r$-neighborhood of $v$ in $T\setminus U$
has at most $(2^{r+1}+1)\varepsilon|T|$ nodes.
\end{proof}

On the other hand, each tree can have only a bounded number of $\varepsilon$-major nodes.

\begin{lemma}
\label{lm-major-bound}
For every $\varepsilon\in (0,1)$,
the number of $\varepsilon$-major nodes of any tree $T$ is at most $\varepsilon^{-2}$.
\end{lemma}

\begin{proof}
If the tree $T$ has at most $\varepsilon^{-1}$ nodes, then there is nothing to prove.
So, we assume that the tree $T$ has more than $\varepsilon^{-1}$ nodes.
We can also assume that $T$ is rooted and all the edges are oriented towards the root.
Let $\alpha(u)$ for a node $u$ of $T$ be the number of nodes in the subtree of $u$ (including $u$ itself) divided by $|T|$.
We claim that if $u$ is $\varepsilon$-major, then $\alpha(u')\le\alpha(u)-\varepsilon$ for every child $u'$ of $u$.
Consider an $\varepsilon$-major node $u$ and suppose that it has a child $u'$ with $\alpha(u')>\alpha(u)-\varepsilon$.
The number of nodes in the component of $T\setminus u$ containing $u'$ is $\alpha(u')|T|$ and
in the component containing the parent of $u$ is $(1-\alpha(u))|T|$.
Since these two components together have more than $(1-\varepsilon)|T|$ nodes, $u$ cannot be $\varepsilon$-major.

Partition the interval $(0,1]$ into $\lceil\varepsilon^{-1}\rceil$ intervals
$J_i=((i-1)\varepsilon,i\varepsilon]\cap (0,1]$, $i\in[\lceil\varepsilon^{-1}\rceil]$.
Note that there is no $\varepsilon$-major node $u$ with $\alpha(u)\in J_1$ since $T$ has more than $\varepsilon^{-1}$ nodes.
Since any two $\varepsilon$-major nodes $u$ and $u'$ with the values $\alpha(u)$ and $\alpha(u')$
from the same interval $J_i$, $i\in[\lceil\varepsilon^{-1}\rceil]$, cannot be joined by an oriented path,
the subtrees of all such nodes are disjoint.
Consequently, the number of $\varepsilon$-major nodes $u$ with $\alpha(u)\in J_i$, $i\in[\lceil\varepsilon^{-1}\rceil]$,
is at most $((i-1)\varepsilon)^{-1}\le \varepsilon^{-1}$ (recall that there are no $\varepsilon$-major node $u$ with $\alpha(u)\in J_1$).
We conclude that the number of $\varepsilon$-major nodes $u$ does not exceed
$(\lceil\varepsilon^{-1}\rceil-1)\varepsilon^{-1}\le\varepsilon^{-2}$.
\end{proof}

We will enhance the signature of plane trees by countably many constants $c_i$, $i\in\NN$, to capture major nodes.
We will require that these constants are interpreted by different nodes of trees, and
some of the constants may not be interpreted at all.
A plane tree with some nodes being constants $c_i$ will be called a {\em plane $c$-tree}.
The considered signature contains countably many constants
but only finitely many constants can be interpreted in a finite plane $c$-tree
since we require that all constants are different.
A sequence of plane $c$-trees $(T_n)_{n\in\NN}$ is {\em null-partitioned}
if the following two conditions hold:
\begin{itemize}
\item if for every $k\in\NN$,
      there exists $n_0\in\NN$ such that all constants $c_1,\ldots,c_k$ are interpreted in $T_n$, $n\ge n_0$, and
\item if for every $\varepsilon>0$, there exist integers $n_0$ and $k_0$ such that
      all the $\varepsilon$-major nodes of every tree $T_n$, $n\ge n_0$, are among the constants $c_1,\ldots,c_{k_0}$.
\end{itemize}
Note that the first condition in the definition above guarantees that 
for every first order formula $\varphi$, there exists $n_0$ such that $\varphi$ can be evaluated in every $T_n$, $n\ge n_0$.

\begin{lemma}
\label{lm-c-trees}
Let $(T_n)_{n\in\NN}$ be a first order convergent sequence of plane trees such that the orders $T_n$ tend to infinity.
There exists a first order convergent null-partitioned sequence of plane $c$-trees $(T'_n)_{n\in\NN}$
obtained from a subsequence of $(T_n)_{n\in\NN}$ by interpreting some of the constants $c_i$.
\end{lemma}

\begin{proof}
We start with constructing first order convergent sequences $(T_n^k)_{n\in\NN}$ of plane $c$-trees, $k\in\NN_0$.
Exactly the constants $c_1,\ldots,c_{2^{2k-1}}$ will be interpreted in $(T_n^k)_{n\in\NN}$ for $k\in\NN$, and
all $2^{-k+1}$-major nodes of every $T_n^k$, $n\in\NN$, will be among the constants.
Let $T^0_n=T_n$ for every $n\in\NN$.
Consider $k\in\NN_0$ and assume that we have constructed the sequence $(T_n^k)_{n\in\NN}$.
Let $n_0$ be such that every $T_n^k$, $n>n_0$, has at least $2^{2k+1}$ nodes.
Assign $2^{-k}$-major nodes of $T_n^k$, $n\in\NN$, that are not already some of the constants,
to different constants $c_{2^{2k-1}+1},\ldots,c_{2^{2k+1}}$;
since the number of $2^{-k}$-major nodes does not exceed $2^{2k}$ by Lemma~\ref{lm-major-bound}, this is possible.
Next assign the constants $c_{2^{2k-1}+1},\ldots,c_{2^{2k+1}}$ that are not interpreted yet
to the remaining nodes of $T_n^k$ in a way that all constants are assigned to different nodes.
Let $S_{n-n_0}$ be the resulting plane $c$-tree.
We set $(T_n^{k+1})_{n\in\NN}$ to be a first order convergent subsequence of this sequence $(S_n)_{n\in\NN}$.

Set $T'_n=T_n^n$ for every $n\in\NN$.
Let $\psi$ be a first order formula and let $i$ be the largest index of a constant $c_i$
appearing in $\psi$ (if none of the constants appears in $\psi$, let $i=1$).
Let $k$ be a positive integer such that $i\le 2^{2k-1}$.
Since the reduct of $(T'_n)_{n=k}^\infty$ obtained by omitting all constants $c_j$, $j>2^{2k-1}$, 
is a subsequence of $(T_n^k)_{n\in\NN}$, and
the sequence $(T_n^k)_{n\in\NN}$ is first order convergent,
the sequence of Stone pairings $\langle\psi,T'_n\rangle$ converges.
Hence, the sequence $(T'_n)_{n\in\NN}$ is first order convergent.

We now show that the sequence $(T'_n)_{n\in\NN}$ is null-partitioned.
Let $\varepsilon>0$ and let $k$ be a positive integer such that $\varepsilon\ge 2^{-k+1}$.
All $2^{-k+1}$-major nodes are among the constants $c_1,\ldots,c_{2^{2k-1}}$ in every tree $T_n^k$, $n\in\NN$.
Set $n_0=k$ and $k_0=2^{2k-1}$.
Since all the $\varepsilon$-major nodes of every tree $T'_n$, $n\ge n_0$, are among the constants $c_1,\ldots,c_{k_0}$ and
the choice of $\varepsilon$ was arbitrary, the sequence $(T'_n)_{n\in\NN}$ is null-partitioned.
\end{proof}

The first order properties are closely linked with Ehrenfeucht-Fra{\"\i}ss\'e games~\cite{bib-ebbinghaus+}.
It is well-known that two structures satisfy the same first order sentences with quantifier depth $d$ if and only if
the duplicator has a winning strategy for the $d$-round Ehrenfeucht-Fra{\"\i}ss\'e game played on two structures.
We slightly alter the standard notion of Ehrenfeucht-Fra{\"\i}ss\'e games to fit our setting of plane $c$-trees
in the way we now present.
The $d$-round Ehrenfeucht-Fra{\"\i}ss\'e game is played by two players, called the spoiler and the duplicator.
In the $k$-th round, the spoiler chooses a node of one of the trees (the spoiler can choose a different tree
in different rounds) and places the $k$-th pebble on that node.
The duplicator responds with placing the $k$-th pebble on a node of the other tree.
At the end of the game,
the duplicator wins
if the mapping between the two plane $c$-trees that
maps the node with the $k$-pebble in the first tree to the node with the $k$-pebble in the other tree
is an isomorphism preserving the relations $\parnt$ and $\succ$ and the first $d$ constants.
The standard argument used in the classical setting yields that
the two plane $c$-trees satisfy the same $d$-sentences if and only if
the duplicator has a winning strategy for the $d$-round Ehrenfeucht-Fra{\"\i}ss\'e game.

The next lemma allows us to prove an extension of Hanf's theorem to plane $c$-trees with unbounded degrees.
Before stating the lemma, we need an additional definition.
The relative position of two nodes $v$ and $v'$ in a plane $c$-tree can be described
by a path in Gaifman's graph between them. Let $u_0\cdots u_k$ be such path, i.e., $u_0=v$ and $u_k=v'$.
This path can be associated with a sequence of words \verb|up|, \verb|down|, \verb|left| and \verb|right| as follows:
$w_i=\verb|up|$ if $\parnt(u_i,u_{i-1})$, $w_i=\verb|down|$ if $\parnt(u_{i-1},u_i)$,
$w_i=\verb|left|$ if $\succ(u_{i-1},u_i)$, and $w_i=\verb|right|$ if $\succ(u_i,u_{i-1})$.
The path is called {\em strongly canonical}
if it is associated with a sequence of the form $\verb|up|^a\verb|left|^b\verb|down|^c$ or $\verb|up|^a\verb|right|^b\verb|down|^c$ with $b>0$, and
it is {\em weakly canonical} if the associated sequence is $\verb|up|^a\verb|down|^b$.
We will refer to a path as {\em canonical} if it is either strongly canonical or weakly canonical.

For $k\in\NN$, the {\em $k$-position} of two nodes $v$ and $v'$ is the sequence of words \verb|up|, \verb|down|, \verb|left| and \verb|right|
associated with the shortest strongly canonical path from $v$ to $v'$ if such a path has length at most $k$;
otherwise, 
if there is a weakly canonical path from $v$ to $v'$ of length at most $k$,
then it is the sequence associated with it, and
if such a path does not exist, then it is undefined.
We say that two pairs of nodes have the same $k$-position
if either their $k$-positions are the same sequence or
they are both undefined.

\begin{lemma}
\label{lm-hanf}
For all pairs of integers $d$ and $\ell$, $0\le\ell\le d$,
there exist integers $\beta_{d,\ell}$ and $\gamma_{d,\ell}$ that satisfy the following.
Suppose that the $d$-round Ehrenfeucht-Fra{\"\i}ss\'e game is played on two (not necessarily finite) plane $c$-trees $T$ and $T'$ and
the first $\ell$ rounds have already been played.
Further suppose that 
if one of the plane $c$-trees $T$ and $T'$ has less than $\gamma_{d,\ell}$ nodes of some $\beta_{d,\ell}$-Hintikka type,
then the number of nodes of this $\beta_{d,\ell}$-Hintikka type in the trees $T$ and $T'$ are the same.
If for every $i$ and $j$, $1\le i,j\le\ell$,
the $\beta_{d,\ell}$-Hintikka types of the nodes with the $i$-th pebble are the same and
the nodes with the $i$-th and the $j$-th pebbles have the same $10^{d-\ell}$-position in $T$ and $T'$,
then the duplicator has a winning strategy.
\end{lemma}

\begin{proof}
Fix $d>0$. We prove the lemma by induction on the number of the remaining rounds of the game, i.e., on $d-\ell$.
We show that the lemma holds with the choice $\beta_{d,\ell}=(10d+12)^{d-\ell+1}$.
During the proof, we will invoke various lower bounds on $\gamma_{d,\ell}$ and
we eventually take $\gamma_{d,\ell}$ to be the largest of these lower bounds.

Suppose that $\ell=d$.
Since the $1$-positions of the pairs of the corresponding pebbled nodes are the same in $T$ and $T'$,
the relations $\parnt$ and $\succ$ induced by the pebbled nodes are isomorphic through the mapping that
maps a pebbled node in $T$ to the corresponding pebbled node of $T'$.
Since the $\beta_{d,d}$-Hintikka types of the nodes with the $i$-th pebble in $T$ and $T'$, $1\le i\le d$, are the same,
the first $d$ constants appear on the corresponding pebbled nodes (if they appear on any of them).
We conclude that the duplicator has won the game.
Note that we can choose $\gamma_{d,\ell}$ arbitrarily, e.g., equal to $1$.

Suppose that $\ell<d$ and that we have proven the existence of $\beta_{d,\ell+1}$ and $\gamma_{d,\ell+1}$.
By symmetry, we can assume that the spoiler has placed the $(\ell+1)$-th pebble on a node $w$ of $T$.
Let $w_1,\ldots,w_{\ell}$ be the nodes with the first $\ell$ pebbles in $T$, and
let $w'_1,\ldots,w'_{\ell}$ be the nodes with the first $\ell$ pebbles in $T'$.
We distinguish several cases based on the existence of a short canonical path from one of the nodes $w_1,\ldots,w_{\ell}$ to $w$.

Suppose first that there exists a canonical path from $w_i$, $1\le i\le\ell$, to $w$ that
has length at most $5\cdot 10^{d-\ell-1}$ and let $W$ be the sequence associated with the path.
Let $W$ be the string of the form $\verb|up|^a\verb|left|^b\verb|down|^c$ or $\verb|up|^a\verb|right|^b\verb|down|^c$ associated with the path.
Among all choices of $w_i$ such that a canonical path from $w_i$ to $w$ of length at most $5\cdot 10^{d-\ell-1}$ exists,
choose the one with $c$ minimal.
We first assume that $c=0$.
Let $w'$ be the node of $T'$ such that there exists a canonical path from $w'_i$ to $w'$ that is also associated with $W$.
The existence of $w'$ follows from that the nodes $w_i$ and $w'_i$ have the same $\beta_{d,\ell}$-Hintikka type and
$\beta_{d,\ell}\ge 5\cdot 10^{d-\ell-1}$.
Note that the node $w'$ is uniquely determined and
that the $10^{d-\ell-1}$-position between $w'_j$ and $w'$ is the same as
the $10^{d-\ell-1}$-position between $w_j$ and $w$ for every $j=1,\ldots,\ell$.
Finally, since $\beta_{d,\ell}\ge 5\cdot 10^{d-\ell-1}+\beta_{d,\ell+1}$,
the $\beta_{d,\ell+1}$-Hintikka types of $w$ and $w'$ are the same.
We conclude that the duplicator will have a winning strategy for the $d-\ell-1$ remaining rounds of the game,
if the duplicator places the $(\ell+1)$-th pebble on $w'$.

We now assume that $c>0$. Note that $c\le 5\cdot 10^{d-\ell-1}$.
Let $z$ be the node at distance $c$ on the path from $w$ to the root and
let $z'$ be the node such that
the $5\cdot 10^{d-\ell-1}$-position of $w_i$ and $z$ is the same as
the $5\cdot 10^{d-\ell-1}$-position of $w'_i$ and $z'$.
Note that $z'$ must exist and is uniquely determined
since the nodes $w_i$ and $w'_i$ have the same $\beta_{d,\ell}$-Hintikka type.
Consider the $(c-1+\beta_{d,\ell+1})$-formula $\psi(u)$ that is satisfied
if a node $u$ has a descendant at depth $c-1$ with the same $\beta_{d,\ell}$-Hintikka type as $w$.
Let $m$ be the number of children $u$ of $z$ satisfying $\psi(u)$ and let $m'$ be the number of such children of $z'$.
Since $\beta_{d,\ell}\ge (10\ell+2\cdot 5)\cdot 10^{d-\ell-1}+\ell+1+\beta_{d,\ell+1}$,
it holds that $m=m'$ or both $m$ and $m'$ are at least $\ell\cdot (10^{d-\ell}+1)+1$.

We say that a child $u$ of $z$ is $w_i$-close 
if there is a canonical path from $w_i$ to $u$ associated with $\verb|up|^{a'}\verb|right|^{b'}$ or
with $\verb|up|^{a'}\verb|left|^{b'}$ for some $a'+b'\le 5\cdot 10^{d-\ell-1}-c$.
Similarly, a child $u$ of $z'$ is $w'_i$-close if such a path exists from $w'_i$ to $u$.
Let $m_i$ be the number of $w_i$-close children $u$ of $z$ satisfying $\psi(u)$, and
let $m'_i$ be the number of $w'_i$-close children $u$ of $z'$ satisfying $\psi(u)$.
Note that every $m_i$ and $m'_i$ is at most $10\cdot 10^{d-\ell-1}+1=10^{d-\ell}+1$.
Since the $10^{d-\ell}$-positions of all the pairs of pebbled nodes are the same in $T$ and $T'$,
$w_i$ is a descendant of $z$ at depth at most $5\cdot 10^{d-\ell-1}$ if and only if $w'_i$ is a descendant of $z'$ at the same depth.
Since $\beta_{d,\ell}\ge 5\cdot 10^{d-\ell-1}+\beta_{d,\ell+1}$ and
$w_i$ and $w'_i$ have the same $\beta_{d,\ell}$-Hintikka type,
it holds $m_i=m'_i$ for every $i=1,\ldots,\ell$.

Let $\tilde m$ be the number of children $u$ of $z$ that satisfy $\psi(z)$ and are $w_i$-close for some $i$, and
let ${\tilde m}'$ be the number of children $u$ of $z'$ that satisfy $\psi(z)$ and are $w'_i$-close for some $i$.
Observe that if the same child $u$ of $z$ is counted both in $m_i$ and $m_j$ for some $i$ and $j$,
then the nodes $w_i$ and $w_j$ are joined by a strongly canonical path of length at most $10^{d-\ell}$ and
the corresponding child of $z'$ is also counted both in $m'_i$ and $m'_j$.
This yields that $\tilde m={\tilde m}'$.
The choice of $c$ as small as possible implies that the node $w$ is not a descendant of a $w_i$-close child of $z$.
Since $\tilde m\le m_1+\cdots+m_{\ell}\le (10^{d-\ell}+1)\ell$, we conclude that $m>\tilde m$ and thus $m'>{\tilde m}'$.
Consequently, there exists a child $u$ of $z'$ satisfying $\psi(u)$ that is not $w'_i$-close for any $i\le\ell$.
Let $w'$ be the descendant of this $u$ at depth $c-1$ that has the same $\beta_{d,\ell+1}$-Hintikka type as $w$.
The duplicator now places the $(\ell+1)$-th pebble on $w'$.
Since the $10^{d-\ell-1}$-position of $w'$ to any $w'_i$, $i\le\ell$ is the same as the $10^{d-\ell-1}$-position of $w$ to $w_i$,
the duplicator has a winning strategy for the remaining $d-\ell-1$ rounds of the game by induction.

It remains to consider the case that there is no canonical path from any $w_i$ to $w$ of length at most $5\cdot 10^{d-\ell-1}$.
If $T'$ has a node $w'$ with the same $\beta_{d,\ell+1}$-Hintikka type as $w$
with no canonical path of length at most $10^{d-\ell-1}$ to any of the nodes $w'_1,\ldots,w'_{\ell}$,
the duplicator places the $(\ell+1)$-th pebble on $w'$ and the existence of the winning strategy follows by induction.
Otherwise, there exist $k_1,\ldots,k_{\ell}$, $0\le k_i\le 10^{d-\ell-1}$, such that
every node of the same $\beta_{d,\ell+1}$-Hintikka type as $w$
is a descendant of $p^{k_i}(w'_i)$ for some $i$, $1\le i\le\ell$ at depth at most $k_i+10^{d-\ell-1}$,
where $p$ is a function that assigns a node its parent.

First suppose that $w$ is at distance less than $2\cdot 10^{d-\ell-1}$ from the root in $T$ and
consider a node $w'$ in $T'$ of the same $\beta_{d,\ell+1}$-Hintikka type as $w$ (which exists since $\gamma_{d,\ell}>0$).
Note that the distance of $w'$ from the root in $T'$ is the same as that of $w$ in $T$.
The node $w'$ is a descendant of $p^{k_i}(w'_i)$ for some $i$, $1\le i\le\ell$.
Hence, $w'_i$ is at distance at most $k_i+2\cdot 10^{d-\ell-1}\le 3\cdot 10^{d-\ell-1}$ from the root.
It follows that $w_i$ is at the same distance from the root in $T$ because they have the same $\beta_{d,\ell}$-Hintikka type.
This implies that there is a canonical path from $w_i$ to $w$ of length at most $5\cdot 10^{d-\ell-1}$,
which contradicts our assumption.

We now assume that $w$ is at distance at least $2\cdot 10^{d-\ell-1}$ from the root.
Let $\psi(u)$ be a local $(2\cdot 10^{d-\ell-1}+\beta_{d,\ell+1})$-formula expressing that
$u$ has a descendant at depth $2\cdot 10^{d-\ell-1}$ of the same $\beta_{d,\ell+1}$-Hintikka type as $w$.
Note that any node of the same $\beta_{d,\ell+1}$-Hintikka type as $w$ (in $T$ or in $T'$)
is at distance at least $2\cdot 10^{d-\ell-1}$ from the root and
must thus be contained in the subtree of a node $u$ satisfying $\psi(u)$ at depth $2\cdot 10^{d-\ell-1}$.
Moreover, every node $u$ of $T'$ that satisfies $\psi(u)$ is one of the nodes $p^{k_i+\kappa}(w'_i)$,
$1\le i\le\ell$ and $0\le\kappa\le 2\cdot 10^{d-\ell-1}$.
Hence, $T'$ contains at most $2\cdot 10^{d-\ell}\cdot\ell$ nodes $u$ satisfying $\psi(u)$.
If $\gamma_{d,\ell}>2\cdot 10^{d-\ell-1}\cdot\ell$, then the number of such nodes $u$ is the same in $T$.
If $\beta_{d,\ell}\ge 5\cdot 10^{d-\ell-1}+\beta_{d,\ell+1}$,
all such nodes $u$ in $T$ are equal to $p^{\kappa}(w_i)$ for some $i$ and $\kappa\le 3\cdot 10^{d-\ell-1}$ (because the same holds in $T'$).
Since every node of $T$ of the same $\beta_{d,\ell+1}$-Hintikka type as $w$
is contained in the subtree of a node $u$ satisfying $\psi(u)$ at depth $2\cdot 10^{d-\ell-1}$,
all nodes of the same $\beta_{d,\ell+1}$-Hintikka type as $w$
are joined to some $w_i$ by a canonical path of length at most $5\cdot 10^{d-\ell-1}$.
In particular, there is a canonical path from $w_i$ to $w$ of length at most $5\cdot 10^{d-\ell-1}$,
which is impossible.
\end{proof}

An immediate corollary of Lemma~\ref{lm-hanf} is the following variant of Hanf's theorem for plane trees.

\begin{theorem}
\label{thm-hanf}
For every integer $d$, there exist integers $D$ and $\Gamma$ such that
for any two (not necessarily finite) plane $c$-trees $T$ and $T'$,
if the trees $T$ and $T'$ have the same number of nodes of each $D$-Hintikka type or
the number of the nodes of this type is at least $\Gamma$ in both $T$ and $T'$,
then the sets of $d$-sentences satisfied by $T$ and $T'$ are the same.
\end{theorem}

We are now ready to distill the essence of first order properties of a first order convergent sequence of plane trees.
Fix a first order convergent null-partitioned sequence of plane $c$-trees $(T_n)_{n\in\NN}$.
We associate the sequence with two functions, $\nu:\FOlocal\to\NN^*$ and $\mu:\FOlocal\to [0,1]$,
which we will refer to as the {\em discrete Stone measure} and the {\em Stone measure} of the sequence (strictly speaking,
the Stone measure as we define it is not a measure in the sense of measure theory,
however, we believe that there is no danger of confusion).
If $\psi\in\FOlocal$, then $\nu(\psi)$ is the limit of the number of nodes $u$ such that $T_n\models\psi(u)$, and
$\mu(\psi)$ is the limit of $\langle\psi,T_n\rangle$.
If $M$ is a plane $c$-tree modeling,
then it is also possible to speak about its discrete Stone measure and its Stone measure
by setting $\nu(\psi)=|\{u\mbox{ s.t. }M\models\psi(u)\}|$ and $\mu(\psi)=\langle\psi,M\rangle$ for $\psi\in\FOlocal$.

Let $\RR$ be the ring formed by finite unions of basic sets of Hintikka chains.
Note that $\RR$ can be equivalently defined as the set containing finite unions of disjoint basic sets of Hintikka chains.
Let $X\in\RR$ be the union of disjoint basic sets corresponding to Hintikka formulas $\psi_1,\ldots,\psi_k$ and
define $\mu_{\RR}(X)=\mu(\psi_1)+\cdots+\mu(\psi_k)$.
Clearly, the mapping $\mu_{\RR}$ is additive.
Since every countable union of non-empty pairwise disjoint sets from $\RR$ that is contained in $\RR$ is finite,
$\mu_{\RR}$ is a premeasure.
By Carath\'eodory's Extension Theorem, the premeasure $\mu_{\RR}$ extends to a measure
on the $\sigma$-algebra formed by Borel sets of Hintikka chains.
We will also use $\mu$ to denote this measure, which is uniquely determined by the Stone measure $\mu$.
Let us remark that the existence of $\mu$ on the $\sigma$-algebra formed by Borel sets of Hintikka chains
can be derived from~\cite[Theorem 8]{bib-pom-unified}
but we have preferred giving a simple direct argument for the clarity and completeness of our exposition.

The following lemma relates first order convergent sequences of plane trees and their limit modelings.
In particular, the lemma reduces the problem of constructing a limit modeling of
a first order convergent null-partitioned sequence of plane $c$-trees,
i.e., a modeling that agrees on all Stone pairings, to constructing a limit modeling that
agrees on Stone pairings involving formulas from $\FOlocal$ only.

\begin{lemma}
\label{lm-modeling}
Let $(T_n)_{n\in\NN}$ be a first order convergent null-partitioned sequence of plane $c$-trees with increasing orders and
let $\nu$ and $\mu$ be its discrete Stone measure and Stone measure, respectively.
If $M$ is a plane $c$-tree modeling such that
\begin{itemize}
\item the discrete Stone measure of $M$ is $\nu$,
\item the Stone measure of $M$ is $\mu$,
%\item the $r$-neighborhood of each node of $M$ is measurable for every $r\in\NN$, and
\item the $r$-neighborhood in $M\setminus\{c_i,i\in\NN\}$ of each node in $M\setminus\{c_i,i\in\NN\}$
      has zero measure for every $r\in\NN$, 
\end{itemize}
then $M$ is a limit modeling of $(T_n)_{n\in\NN}$.
\end{lemma}

\begin{proof}
By Theorem~\ref{thm-hanf},
the discrete Stone measure fully determines which first order sentences are satisfied by $M$,
in particular, $\langle\psi,M\rangle$ is equal to the limit of $\langle\psi,T_n\rangle$ for every first order sentence $\psi$.
The situation is trickier for first order formulas $\psi$ with free variables.

Fix such a formula $\psi$ with $k>0$ free variables and with quantifier depth $d$.
Let $\beta$ and $\gamma$ be the quantities from Lemma~\ref{lm-hanf} applied for $d$ and $\ell=0$, and
let $B$ be the number of $\beta$-Hintikka formulas.
Lemma~\ref{lm-hanf} yields that
the truth value of $\psi$ for a particular evaluation of its free variables is determined
by the $\beta$-Hintikka types of the values of the free variables and
by the statistics $\beta$-Hintikka types (possibly truncated at the value of $\gamma$) of all the nodes,
if none of the pairs of the values of the free variables is joined by a canonical path of length at most $10^d$.
If the distance between the values of the free variables is smaller, then their distance can also come into the play.
Since there exist first order convergent sequences of plane $c$-trees such that
the probability of two random nodes being joined by a canonical path of length at most $10^d$ is
bounded away from zero for every tree in the sequence,
we have to account for this possibility, and we do so using the major nodes we have introduced.

Fix $\varepsilon>0$.
Since the sequence $(T_n)_{n\in\NN}$ is null-partitioned,
there exists $n_0$ such that all the $\varepsilon$-major nodes of each of the trees $T_n$, $n\ge n_0$, are among the constants $c_1,\ldots,c_{k_0}$.
We assume that $k_0$ and $n_0$ are large enough such that
\begin{itemize}
\item the largest index $i$ of a constant $c_i$ that appears in $\psi$ does not exceed $k_0$,
\item all trees $(T_n)_{n\in\NN}$, $n\ge n_0$, contain the same set of the constants among $c_1,\ldots,c_{k_0}$ (the existence of such $n_0$
follows from the first order convergence of the sequence $(T_n)_{n\in\NN}$), and
\item $|\langle \varphi,T_n\rangle-\langle \varphi,M\rangle|\le\varepsilon$
for every $\beta$-Hintikka formula $\varphi$ and $n\ge n_0$ (the existence of such $n_0$
follows from the fact that are only $B$ different $\beta$-Hintikka formulas and the sequence $(T_n)_{n\in\NN}$ is first order convergent).
\end{itemize}
\noindent
The probability that one of the $k$ randomly chosen nodes of $T_n$, $n\ge n_0$, is among the constants $c_1,\ldots,c_{k_0}$ or
two such nodes are at distance at most $2\cdot 10^d$ (in $T_n\setminus\{c_1,\ldots,c_{k_0}\}$)
is at most $kk_0|T_n|^{-1}+(2^{2\cdot 10^d}+1)k^2\varepsilon$,
since every $2\cdot 10^d$-neighborhood of a node in the tree $T_n\setminus\{c_1,\ldots,c_{k_0}\}$
has at most $(2^{2\cdot 10^d+1}+1)\varepsilon|T_n|$ nodes by Lemma~\ref{lm-major-neighborhood}.

Let us now look at the modeling $M$.
The $2\cdot 10^d$-neighborhood of any node in $M\setminus\{c_i,i\in\NN\}$ has zero measure
but the same need not be true for the $2\cdot 10^d$-neighborhoods in $M\setminus\{c_1,\ldots,c_{k_0}\}$.
However,
since it is possible to express by a local first order formula that
a node is at distance at most $4\cdot 10^{d}$ from a constant $c_i$ and
the $4\cdot 10^{d}$-neighborhood of each node $c_i$, $i>k_0$, contains at most $(2^{4\cdot 10^{d}+1}+1)\varepsilon|T_n|$ nodes
in $T_n\setminus\{c_1,\ldots,c_{k_0}\}$, $n\ge n_0$, by Lemma~\ref{lm-major-neighborhood},
the measure of the $4\cdot 10^{d}$-neighborhood of $c_i$, $i>k_0$, is at most $(2^{4\cdot 10^{d}+1}+1)\varepsilon$.
Let $u$ be a node of $M\setminus\{c_1,\ldots,c_{k_0}\}$.
If the $2\cdot 10^d$-neighborhood of $u$ in $M\setminus\{c_1,\ldots,c_{k_0}\}$ does not contain any of the constants $c_i$, $i>k_0$,
then its measure is zero by the assumption of the lemma.
If the $2\cdot 10^d$-neighborhood of $u$ in $M\setminus\{c_1,\ldots,c_{k_0}\}$ contains one of the constants $c_i$, $i>k_0$,
it is contained in the $4\cdot 10^{d}$-neighborhood of $c_i$ in $M\setminus\{c_1,\ldots,c_{k_0}\}$ and
its measure is at most $(2^{4\cdot 10^{d}+1}+1)\varepsilon$,
which is the upper bound on the measure of the $4\cdot 10^{d}$-neighborhood of $c_i$.
We conclude that the probability that any two of $k$ randomly chosen nodes of $M\setminus\{c_1,\ldots,c_{k_0}\}$
are at distance at most $2\cdot 10^d$ in $M\setminus\{c_1,\ldots,c_{k_0}\}$ is at most $(2^{4\cdot 10^{d}+1}+1)k^2\varepsilon$.

For any $k$-tuple of $\beta$-Hintikka formulas $(\varphi_1,\ldots,\varphi_k)$,
the probabilities that a random $k$-tuple of nodes of $T_n$, $n\ge n_0$, and
a random $k$-tuple of nodes of $M$ have $\beta$-Hintikka types containing $\varphi_1,\ldots,\varphi_k$ (in this order)
differ by at most $k\varepsilon$
since $|\langle \varphi,T_n\rangle-\langle \varphi,M\rangle|\le\varepsilon$ for every $\beta$-Hintikka formula $\varphi$.
Note that there are $B^k$ choices of $k$-tuples of $\beta$-Hintikka formulas.
Also note that if two nodes are at distance larger than $2\cdot 10^d$ in $T\setminus\{c_1,\ldots,c_{k_0}\}$ and
have the same $\beta_d$-Hintikka type,
then either they are not joined by a canonical path of length at most $10^d$ or
their $10^d$-position is uniquely determined by their $\beta_d$-Hintikka types (since a canonical path of length at most $10^d$
must pass through one of the constants $c_1,\ldots,c_{k_0}$).
The same holds for any two nodes of $M\setminus\{c_1,\ldots,c_{k_0}\}$.
Hence, the Stone pairings $\langle \psi,T_n\rangle$ and $\langle\psi,M\rangle$ for $n\ge n_0$,
which are the probabilities that a random $k$-tuple of nodes of $T_n$, $n\ge n_0$, and
a random $k$-tuple of nodes of $M$ satisfy $\psi$, differ by at most
$$B^k k\varepsilon+\frac{kk_0}{|T_n|}+(2^{2\cdot 10^{d}+1}+1)k^2\varepsilon+(2^{4\cdot 10^{d}+1}+1)k^2\varepsilon\;\mbox{.}$$
Since the choice of $\varepsilon$ was arbitrary and the orders $|T_n|$ of the trees $T_n$ tend to infinity,
we conclude that for every $\varepsilon_0>0$, there exists $n_0$ such that
$|\langle \psi,T_n\rangle-\langle\psi,M\rangle|\le\varepsilon_0$ for every $n\ge n_0$.
Since this holds for every first order formula $\psi$, the modeling $M$ is a limit modeling of $(T_n)_{n\in\NN}$.
\end{proof}

\subsection{Composition}
\label{sect-compose}

In this section, we complement Lemma~\ref{lm-modeling} by constructing a modeling with the properties stated in that lemma.

\begin{lemma}
\label{lm-composition}
Let $(T_n)_{n\in\NN}$ be a first order convergent null-partitioned sequence of plane $c$-trees and
let $\nu$ and $\mu$ be its discrete Stone measure and Stone measure, respectively.
There exists a plane $c$-tree modeling such that
\begin{itemize}
\item the discrete Stone measure of $M$ is $\nu$,
\item the Stone measure of $M$ is $\mu$,
%\item the $r$-neighborhood of each node of $M$ is measurable for every $r\in\NN$,
\item the $r$-neighborhood of each node in $M\setminus\{c_i,i\in\NN\}$ has zero measure for every $r\in\NN$, and
\item the modeling $M$ satisfies the strong finitary mass transport principle.
\end{itemize}
\end{lemma}

\begin{proof}
We first extend the mapping $\nu$ to Hintikka chains as follows.
If $\Psi=(\psi_i)_{i\in\NN}$ is a Hintikka chain, then $$\nu(\Psi)=\lim_{i\to\infty}\nu(\psi_i)\;\mbox{.}$$
Observe that the support of the measure $\mu$ is a subset of $\nu^{-1}(\infty)$.
Indeed, if $\nu(\Psi)=k\in\NN$, then there exists a $d$-Hintikka chain $\psi_d$ such that $\nu(\psi_d)=k$.
Hence, there exists $n_0$ such that every $T_n$, $n\ge n_0$ contains exactly $k$ nodes $u$ such that $T_n\models\psi_d(u)$.
Since the sequence $(T_n)_{n\in\NN}$ is null-partitioned (in particular,
the orders of $(T_n)_{n\in\NN}$ tend to infinity), it follows that $\mu(\psi_d)=0$.
Consequently, the measure of the basic set of Hintikka chains corresponding to $\psi_d$ is zero and
$\Psi$ is not in the support of $\mu$.

The node set of the modeling $M$ that we construct consists of two sets:
$V_f$ contains all pairs $(\Psi,i)$ where $\Psi$ is a Hintikka chain such that $\nu(\Psi)\in\NN$ and $i\in [\nu(\Psi)]$, and
$V_{\infty}=\nu^{-1}(\infty)\times [0,1)^3$.
Note that the set $V_f$ is countable.

A Hintikka chain encodes many properties of a node. In particular, it uniquely determine the Hintikka chains of
the parent and the successor (if they exist) and we will refer to these chains as the {\em parent Hintikka chain} and
the {\em successor Hintikka chain}.
The Hintikka chain of a node also determines whether the node is one of the constants,
whether it is the root, and
how many children satisfying a particular local first order formula the node can have.
In a slightly informal way,
we will be speaking about these properties by saying that the node of the Hintikka chain is a constant,
it is the root, etc.

We now continue with the construction of the modeling $M$ with setting the constants.
For each constant $c_i$,
there is at most one Hintikka chain $\Psi$ with $\nu(\Psi)>0$ such that the node of $\Psi$ is the constant $c_i$.
If such a Hintikka chain $\Psi$ exists, then $\nu(\Psi)=1$ and we set the constant $c_i$ to be the node $(\Psi,1)\in V_f$.

We next define the child-parent relation $\parnt$ and the successor relation $\succ$.
Let $(\Psi,i)\in V_f$.
If the node of the Hintikka chain $\Psi$ is the root, then $(\Psi,i)$ has no parent and no successor.
Otherwise, let $\Psi'$ be the parent Hintikka chain of $\Psi$.
The definition of the discrete Stone measure and the first order convergence of $(T_n)_{n\in\NN}$ imply that
$\nu(\Psi')$ is a non-zero integer which divides $\nu(\Psi)$.
The parent of the node $(\Psi,i)$ is the node $(\Psi',i')$ where $i'=\lceil i\nu(\Psi')/\nu(\Psi) \rceil$.
If the node of the Hintikka chain $\Psi$ has no successor, then $(\Psi,i)$ has no successor.
Otherwise, let $\Psi''$ be the successor Hintikka chain of $\Psi$.
Since $\nu$ is a discrete Stone measure of a first order convergent sequence, it must hold that $\nu(\Psi)=\nu(\Psi'')$.
We set the successor of $(\Psi,i)$ to be the node $(\Psi'',i)$.

Let $(\Psi,h,s,t)\in V_{\infty}$.
Since every tree has a unique root and $\nu(\Psi)=\infty$, the node of the Hintikka chain $\Psi$ is not the root.
Let $\Psi'$ be the parent Hintikka chain.
If $\nu(\Psi')$ is a non-zero integer, the parent of $(\Psi,h,s,t)$ is the node $(\Psi',\lfloor t\nu(\Psi')\rfloor+1)$.
If the node of the Hintikka chain $\Psi$ has no successor, then $(\Psi,h,s,t)$ has no successor.
Otherwise, let $\Psi''$ be the successor Hintikka chain.
Since $\nu$ is a discrete Stone measure of a first order convergent sequence, it must hold that $\nu(\Psi)=\nu(\Psi'')$, and
we can set the successor of the node $(\Psi,h,s,t)$ to be the node $(\Psi'',h,s+\sqrt{2}\mod 1,t)$.

It remains to consider the case
that $\nu(\Psi')$ is equal to $\infty$ (it cannot be equal to zero since $\nu$ is a discrete Stone measure of a first order convergent sequence).
If the node of the Hintikka chain $\Psi'$ has only finitely many, say $m$, children with the Hintikka chain $\Psi$,
the parent of $(\Psi,h,s,t)$ is the node $(\Psi',h+\sqrt{2}\mod 1,ms\mod 1,t)$.
This type of child-parent edges will be referred to as {\em finitary arcs} in the proof of Theorem~\ref{thm-pw}.
We also have to define the successor of $(\Psi,h,s,t)$.
If the node of $\Psi$ has no successor, then $(\Psi,h,s,t)$ has no successor.
Otherwise, let $\Psi''$ be the successor Hintikka chain.
Since $\nu$ is a discrete Stone measure of a first order convergent sequence, it must hold that $\nu(\Psi)=\nu(\Psi'')$.
We set the successor of the node $(\Psi,h,s,t)$ to be the node $(\Psi'',h,s,t)$ (note that $\Psi\not=\Psi''$ since
the node of the Hintikka chain $\Psi'$ has exactly $m$ children with the Hintikka chain $\Psi$).

Finally, if the node of the Hintikka chain $\Psi'$ can have an infinite number of children with the Hintikka chain $\Psi$,
the parent of $(\Psi,h,s,t)$ is the node $(\Psi',h+\sqrt{2}\mod 1,\zeta_1(t),\zeta_2(t))$
where the functions $\zeta_1,\zeta_2:[0,1)\to [0,1)$ are defined as follows.
For $x\in[0,1)$, there exist unique numbers $x_i\in\{0,1\}$, $i\in\NN$ such that $\{i:x_i=0\}$ is infinite and
$x=\sum_{i=1}^{\infty} x_i2^{-i}$. The values $\zeta_1(x)$ and $\zeta_2(x)$ are defined as follows.
$$\zeta_1(x)=\sum_{i=1}^{\infty} x_{2i-1}2^{-i} \mbox{ and } \zeta_2(x)=\sum_{i=1}^{\infty} x_{2i}2^{-i} \;\mbox{.}$$
Note that the mapping $\zeta=(\zeta_1,\zeta_2)$ is an invertible measure preserving transformation from $[0,1)$ to $[0,1)^2$.
The just defined child-parent edges will be referred to as {\em infinitary arcs} in the proof of Theorem~\ref{thm-pw}.
If the node of $\Psi$ has no successor, then $(\Psi,h,s,t)$ has no successor.
Otherwise, let $\Psi''$ be the successor Hintikka chain.
Since $\nu$ is a discrete Stone measure of a first order convergent sequence, it must hold that $\nu(\Psi)=\nu(\Psi'')$, and
we set the successor of the node $(\Psi,h,s,t)$ to be the node $(\Psi'',h,s+\sqrt{2}\mod 1,t)$.

We next define a probability measure on $V_f\cup V_{\infty}$ as follows.
Every subset of $V_f$ is measurable and its measure is zero.
The set $V_{\infty}$ is equipped with the product measure determined by $\mu$ and the uniform (Borel) measure on $[0,1)^3$.
We need to verify that the constructed plane $c$-tree $M$ with this probability measure is a modeling and
that $M$ satisfies the properties given in the lemma.
The construction of $M$ implies that,
for every Hintikka chain $\Psi=(\psi_d)_{d\in\NN}$,
the nodes $u$ of $M$ such that the $d$-th Hintikka type of $u$ is $\psi_d$ for all $d\in\NN$
are exactly the nodes $(\Psi,i)\in V_f$ or the nodes $(\Psi,h,s,t)\in V_\infty$.
This yields that the discrete Stone measure and the Stone measure of $M$ are $\nu$ and $\mu$, respectively.

Since the mappings $f(x):=x+\sqrt{2}\mod 1$ and $\zeta$ are invertible measure preserving transformations and
$\mu$ is the Stone measure of a first order-convergent sequence of plane $c$-trees, 
one can argue using the way that $M$ was constructed that
all first order definable subsets of $M^k$, $k\in\NN$, are measurable in the product measure and
the modeling $M$ satisfies the strong finitary mass transport principle.
This argument follows the lines of the analogous argument in the proof of Lemma 39 in~\cite{bib-pom-trees}, and
so we briefly sketch the main steps of the argument.
First, every subset of $M$ defined by a local first order formula $\psi$ is open (and so it is measurable)
since it is the set of all nodes with the Hintikka chain in the first coordinate consistent with $\psi$.
Next, the subset of $M^2$ given by the child-parent relation and the subset of $M^2$ given by the successor relation are measurable.
Lemma~\ref{lm-hanf} yields that whether a $k$-tuple of nodes of $M$ satisfies a particular formula with $k$ variables
depends on the types of the nodes and the configuration formed by them,
which can be described by a finite encoding using the child-parent relation and the successor relation in a measurable way.
It follows that every first order-definable subset of $M^k$ is measurable, i.e., $M$ is a modeling.
Finally, the modeling $M$ satisfies the strong finitary mass transport principle
since the mappings $f(x):=x+\sqrt{2}\mod 1$ and $\zeta$ are invertible measure preserving transformations, and
$\mu$ is the Stone measure of a first order-convergent sequence of plane $c$-trees (this is important
in the analysis of the finitary and infinitary child-parent edges).

It now remains to verify the third property from the statement of the theorem.
%First observe, the $r$-neighborhood of any node of $M$ in $M$ is measurable for every $r\in\NN$ and
%the same is true for the $r$-neighborhoods in $M\setminus\{c_i,i\in\NN\}$.
Suppose that there exist a node $(\Psi,i)\in V_f$ and $r\in\NN$ such that
the $r$-neighborhood of $(\Psi,i)$ in $M\setminus\{c_i,i\in\NN\}$ has a positive measure, say $\varepsilon>0$.
Let $\Psi=(\psi_j)_{j\in\NN}$.
Since $\nu(\Psi)$ is finite, there exists $\psi_j$ such that $\nu(\psi_j)=m$ and $m\in\NN$.
By the definition of a null-partitioned sequence of trees,
there exist integers $n_0$ and $k_0$ such that
the $r$-neighborhood of each node in $T_n\setminus\{c_1,\ldots,c_{k_0}\}$, $n\ge n_0$, contains at most $\varepsilon |T_n|/2m$ nodes.
In particular, this holds for the $m$ nodes satisfying $\psi_j$.
This implies that the Stone measure of nodes joined to one of the $m$ nodes satisfying $\psi_j$
by a path of length at most $r$ that avoids $c_1,\ldots,c_{k_0}$ does not exceed $\varepsilon/2$ (note that this
property is first order expressible and therefore captured by the Stone measure).
Hence, the $r$-neighborhood of $(\Psi,i)$ in $M\setminus\{c_i,i\in\NN\}$ can have measure at most $\varepsilon/2$,
contrary to our assumption that its measure is $\varepsilon$.

Next, suppose that there exist a node $(\Psi,h,s,t)\in V_\infty$ and an integer $r$ such that
the $r$-neighborhood of $(\Psi,h,s,t)$ in $M\setminus\{c_i,i\in\NN\}$ has a positive measure.
The $r$-neighborhood of $(\Psi,h,s,t)$ contains only nodes of $V_\infty$ (otherwise,
$V_f$ would contain a node  whose $2r$-neighborhood has a positive measure).
However, the definition of the modeling $M$ implies that the second coordinate of the nodes
in the $r$-neighborhood of $(\Psi,h,s,t)$ are $h\pm\sqrt{2}i\mod 1$ for $i=-r,\ldots,+r$.
Since the set of all nodes of $V_{\infty}$ with the second coordinate equal to $h\pm\sqrt{2}i\mod 1$, $i=-r,\ldots,+r$,
has measure zero, the $r$-neighborhood of $(\Psi,h,s,t)$ in $M\setminus\{c_i,i\in\NN\}$ also has measure zero.
\end{proof}

Since Theorem~\ref{thm-plane} is easy to prove in the case
when the orders of the plane trees in a first order convergent sequence $(T_n)_{n\in\NN}$ do not tend to infinity (in such case,
there exists $n_0$ such that all trees $T_n$, $n\ge n_0$, are isomorphic),
Lemmas~\ref{lm-c-trees}, \ref{lm-modeling} and \ref{lm-composition} yield Theorem~\ref{thm-plane}.

\section{Interpretation schemes}
\label{sect-scheme}

Interpretation schemes can be used to translate results on the existence of modelings from one class of discrete structures to another.
We now recall some of the results from~\cite[Sections 6--8]{bib-pom-unified}, which we formulate in the setting that we need in our exposition.
Let $\kappa$ and $\lambda$ be two signatures where $\lambda$ has $k$ relational symbols $R_1,\ldots,R_k$ with arities $r_1,\ldots,r_k$.
A (first order) {\em interpretation scheme} $I$ of $\lambda$-structures in $\kappa$-structures
consists of a first order $\kappa$-formula $\Phi_0$ with one free variable and
first order $\kappa$-formulas $\Phi_i$, each with $r_i$ free variables.
If $K$ is a $\kappa$-structure, then $I(K)$ is the $\lambda$-structure with the domain $\{x\in K:\Phi_0(x)\}$ and
$R_i=\{(x_1,\ldots,x_{r_i})\in K^{r_i}:\Phi_i(x_1,\ldots,x_{r_i})\}$.
We say that the relational symbol $R_i$ is {\em interpreted trivially} by a relational symbol $Q$ (from the signature $\kappa$)
if $\Phi_i(x_1,\ldots,x_{r_i})=Q(x_1,\ldots,x_{r_i})$.

The following lemma, which links modelings of $\kappa$-structures and $\lambda$-structures,
summarizes the results of~\cite[Sections 6--8]{bib-pom-unified}.

\begin{lemma}
\label{lm-scheme}
Let $\kappa$ and $\lambda$ be two signatures and $I=\{\Phi_0,\Phi_1,\ldots\}$ an interpretation scheme of $\lambda$-structures in $\kappa$-structures.
Suppose that $(L_n)_{n\in\NN}$ is a first order convergent sequence of $\lambda$-structures and
that there exists a first order convergent sequence $(K_n)_{n\in\NN}$ of $\kappa$-structures such that $I(K_n)=L_n$.
If the sequence $(K_n)_{n\in\NN}$ has a limit modeling $M$ with a measure $\mu$ and $\langle\Phi_0,M\rangle$ is positive,
then the $\lambda$-structure $I(M)$ together
with the measure $\mu'$ defined on the measurable subsets $Z$ of $\{x\in M:\Phi_0(x)\}$ as $\mu'(Z)=\mu(Z)/\langle\Phi_0,M\rangle$
is a limit modeling for the sequence $(L_n)_{n\in\NN}$.

Moreover, if $\kappa$ and $\lambda$ are extensions of the signature of graphs,
the edge relational symbol from $\lambda$ is interpreted trivially by the edge relational symbol from $\kappa$, and
the modeling $M$ satisfies the strong finitary mass transport principle,
then the modeling $I(M)$ also satisfies the strong finitary mass transport principle.
\end{lemma}

As a simple application of the machinery that we have introduced,
we extend results of Section~\ref{sect-plane} to plane forests with nodes colored with finitely many colors.
A $k$-colored plane forest is a graph where each component is a plane tree and each node  is assigned one of the $k$ colors.
The signature we use to describe $k$-colored plane forests contains the two binary relational symbols $\parnt$ and $\succ$ and
$k$ unary relational symbols used to described the colors of the nodes.

We introduce a constructive mapping $t$ that maps a $k$-colored plane forest $F$ to a plane tree.
The plane tree $t(F)$ is constructed as follows:
first introduce a new node $r$ and make the root of each tree in $F$ a child of $r$.
The order of the children of $r$ is set arbitrarily.
Next, each node of $F$ colored with $i\in [k]$ gets $i$ new children
which appear first in the linear order of its children.

We define an interpretation scheme $I$ of $k$-colored plane forests in plane trees such that
$I(t(F))=F$ for every $k$-colored plane forest $F$.
The formula $\Phi_0(x)$ is satisfied if $x$ is neither the root nor a leaf (note that each node of $F$ has been added at least one child).
The parent and the successor relation are simply interpreted by the formulas $\Phi_1(x,y)=\parnt(x,y)$ and $\Phi_2(x,y)=\succ(x,y)$.
Finally, the formula $\Phi_{i+2}(x)$ determining whether the node $x$ has the color $i$
expresses whether the $i$-th child of $x$ exists and if so, whether it is a leaf.
It is straightforward to check that $I(t(F))=F$ for every $k$-colored plane forest $F$ as desired.

Theorem~\ref{thm-plane} and Lemma~\ref{lm-scheme} yield the following.

\begin{theorem}
\label{thm-omega}
Every first-order convergent sequence $k$-colored plane forests has a limit modeling
satisfying the strong finitary mass transport principle.
\end{theorem}

\begin{proof}
Let $(F_n)_{n\in\NN}$ be a first order convergent sequence of $k$-colored plane forests, and
let $(T_n)_{n\in\NN}$ be a first order convergent subsequence of the sequence $(t(F_n))_{n\in\NN}$,
which exists by the compactness.
By Theorem~\ref{thm-plane}, the sequence $(T_n)_{n\in\NN}$ has a limit modeling $M$ which satisfies the strong finitary mass transport principle.
By Lemma~\ref{lm-scheme}, the modeling $I(M)$ is a limit modeling of the sequence $I(T_n)$,
which is a subsequence of the sequence $(F_n)_{n\in\NN}$
Since $(F_n)_{n\in\NN}$ is first order convergent, $M$ is also a modeling of $(F_n)_{n\in\NN}$.
The modeling $M$ satisfies the strong finitary mass transport principle by Lemma~\ref{lm-scheme}.
\end{proof}

\section{Graphs with bounded path-width}
\label{sect-pw}

In this section, we present an interpretation scheme of graphs with bounded path-width in colored plane forests.
The core of the scheme is a recursive construction parameterized by the path-width.
However, we need to introduce additional definitions to those given in Section~\ref{subsect-pw}.
If $A$ is a finite set of integers,
then an {\em $A$-interval graph} is a semi-interval graph such that
each interval is colored with an element of $A$ and
no two intersecting intervals have the same color (an example is given in Figure~\ref{fig-pw1}).
Observe that the path-width of an $A$-interval graph is at most $|A|-1$.
On the other hand, every graph with path-width $p$ can be equipped with a semi-interval representation and a coloring
to become a $[p+1]$-interval graph.

We construct a mapping $t_A$ that assigns an $A$-interval graph $G$ a $(|A|\cdot 2^{|A|+2})$-colored plane tree $T$ of depth at most $|A|$.
The construction of the mappings $t_A$ is recursive based on $|A|$ and $|G|$.
The graph $G$ will be mapped to a tree $T$ such that following is satisfied.
Each vertex of $G$ will be associated with at least one node of $T$, and
each node of $T$ different from the root will be associated with a vertex of $G$.
The colors of the nodes of $T$ will determine the structure of $G$ in such a way
it will be possible to reconstruct the graph $G$ by a first order interpretation.
All the nodes except for the root will be colored and
the colors of the nodes will be from the set $A\times 2^A\times 2^{\{\rightarrow,\leftarrow\}}$.
Examples, which can be instructive to be looking at while reading a formal description that we now give,
can be found in Figures~\ref{fig-pw2} and~\ref{fig-pw3}.

\begin{figure}
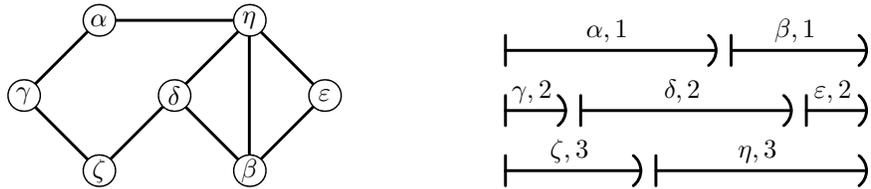

\begin{center}
\epsfbox{folim-pw.1}
\hskip 2cm
\epsfbox{folim-pw.2}
\end{center}
\caption{A $[3]$-interval graph $G$. The graph is in the left and its semi-interval representation in the right.
         Greek letters are used to denote the vertices and integers their colors.}
\label{fig-pw1}
\end{figure}

\begin{figure}
\begin{center}
\epsfbox{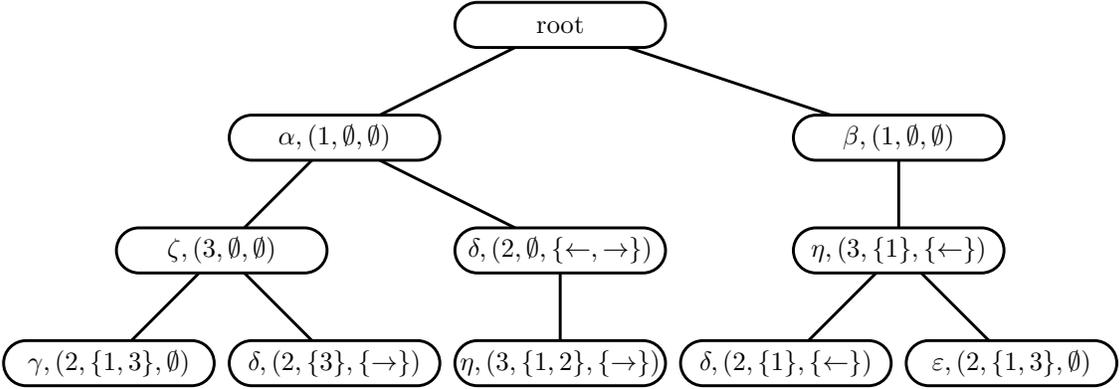}
\end{center}
\caption{The tree $t_{[3]}(G)$ of the $[3]$-interval graph $G$ depicted in Figure~\ref{fig-pw1}.
	 Each node of the tree except for the root is labelled with the associated vertex and the triple that it is colored with.}
\label{fig-pw2}
\end{figure}

\begin{figure}
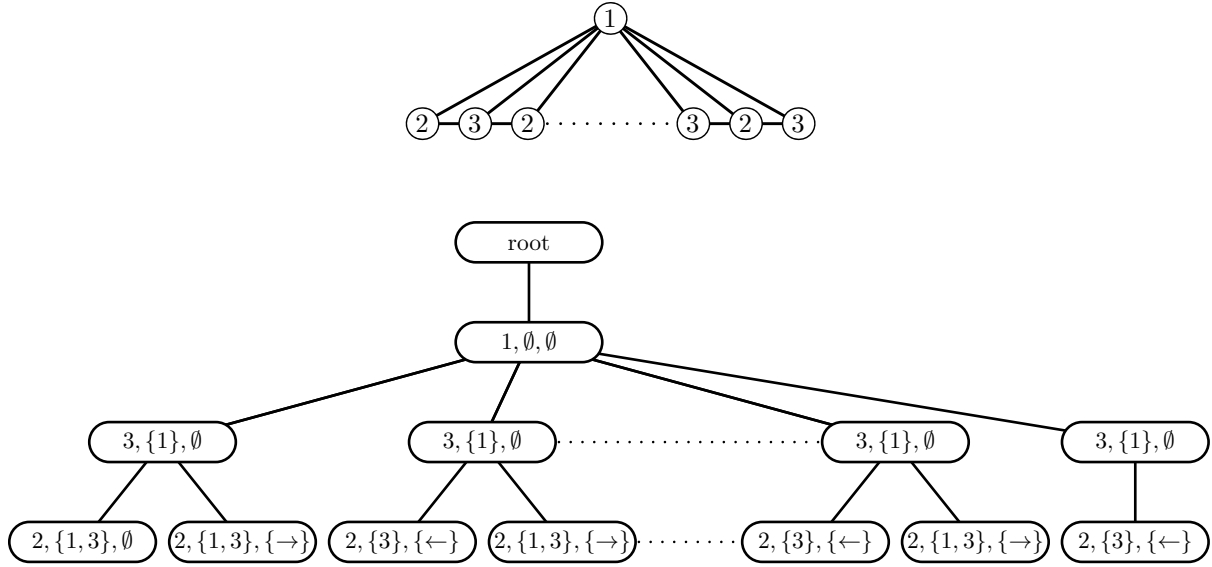

\begin{center}
\epsfbox{folim-pw.4}
\vskip 10mm
\epsfxsize 16cm
\epsfbox{folim-pw.5}
\end{center}
\caption{A fan graph $G$ (in the top) and the tree $t_{[3]}(G)$ (in the bottom).
         Each vertex of the fan is labelled with its color and
         each non-root node of the tree is labelled with the triple that it is colored with.}
\label{fig-pw3}
\end{figure}

It is convenient in our recursive construction to sometimes exhaust the graph so it becomes empty;
the empty graph will be mapped to the tree consisting of a single node.
We fix a non-empty graph $G$ for the remainder of the definition.

Let $v$ be a vertex associated with a longest interval intersecting the first segment,
let $x$ be the color of $v$, and let $A'=A\setminus\{x\}$.
Further let $G'$ be the $A'$-interval graph obtained by restricting $G\setminus\{v\}$ to the vertices
with the intervals intersecting the interval of $v$ and
by restricting the interval representation to the segments contained in the interval of $v$.
Finally,
let $G''$ be the $A$-interval graph obtained by restricting $G$ to the vertices with the intervals intersecting
the complement of the interval of $v$,
by restricting the interval representation to the segments not contained in the interval of $v$, and
by removing all the edges between the vertices of $G''$ that are also contained in $G'$.
Note that if the interval of $v$ intersects all segments, then the graph $G''$ is empty.
The tree $T=t_A(G)$ is obtained in the following way:
the root of the tree $t_{A'}(G')$ becomes the first child of the root $t_A(G'')$ and
is colored with $(x,\emptyset,\emptyset)$.
The color $x$ is inserted into the second coordinate of each node of $t_{A'}(G')$ that
is associated with a vertex adjacent to $v$ and such that
the third coordinate of its color does not contain $\leftarrow$.
Finally,
$\rightarrow$ is inserted to the third coordinate of each node of $t_{A'}(G')$ associated with a vertex in $V(G')\cap V(G'')$, and
$\leftarrow$ is inserted to the third coordinate of each node of $t_{A}(G'')$ associated with a vertex in $V(G')\cap V(G'')$.

We summarize key properties of the mappings $t_A$, which we need further, in the next lemma.

\begin{lemma}
\label{lm-pw}
Let $G$ be an $A$-interval graph and let $T=t_A(G)$. The following holds.
\begin{itemize}
\item Each vertex of $G$ with its interval intersecting the first segment is associated with a single node of $T$, and
      this node is one of the nodes on the path from the root going all the time to the first child of a node.
\item Each vertex of $G$ with its interval intersecting the last segment is associated with at most $|A|$ nodes of $T$.
\item Each vertex of $G$ is associated with at most $|A|+1$ nodes.
\item Each vertex of $G$ is associated with exactly one node such that the third coordinate of its color does not contain $\leftarrow$.
\item For every edge $vv'$ of the graph $G$,
      there exists a unique pair $u$ and $u'$ of nodes of $T$ that are associated with $v$ and $v'$ (not necessarily in this order) such that
      $u$ is in the subtree of $u'$, the color of $u$ is $(x,X,Z)$, the color of $u'$ is $(x',X',Z')$ and $x'\in X$.
\end{itemize}
\end{lemma}

\begin{proof}
The proof proceeds by induction on $|A|$ and $|G|$.
If $G$ has a vertex $v$ with its interval intersecting all segments,
then $v$ is associated with a single node of $t_A(G)$.
The statements for other vertices of $G$
follow by induction applied to the graph $G\setminus\{v\}$ in the definition of $t_A$ (note that
node associated with $v$ is the only child of the root of $T$ and
the edges incident with $v$ will be represented in the way described in the last point of the lemma).

In the remainder of the proof,
we assume $G$ has no such vertex and let $G'$ and $G''$ be the $A'$-interval and $A$-interval graphs as in the definition of $t_A(G)$.
Since each edge of $G$ is either in $G'$ or in $G''$ or
it is incident with the vertex associated with the first child of the root of $t_A(G)$,
the fifth statement in the lemma follows. We now verify the remaining four statements. 

Consider a vertex $v$ of $G$. 
If the interval of $v$ intersects the first segment,
then $v$ is associated with the first child of the root of $t_A(G)$ or it belongs to $G'$ (and not $G''$).
If $v$ is associated with the root, then the first and the forth statements clearly hold,
and if $v$ belongs to $G'$,
then the first and the forth statements follow by induction applied to the $A'$-interval graph $G'$.

Next, suppose that the interval of $v$ intersects the last segment.
If the vertex $v$ belongs to $G'$,
then it is associated with at most $|A'|=|A|-1$ nodes of $t_{A'}(G')$ and a single node of $t_A(G'')$, and
only one of the nodes of $t_{A'}(G')$ has the property that the third coordinate of its color does not contain $\leftarrow$.
If the vertex $v$ does not belong to $G'$,
then it is associated with at most $|A|$ nodes of $t_{A}(G'')$ and
exactly one of these nodes has the property that the third coordinate of its color does not contain $\leftarrow$.

Finally,
if the interval of $v$ does not intersect the first or the last segment,
then $v$ is associated with at most $|A'|+1=|A|$ nodes of $t_{A'}(G')$ if $v$ belongs to $G'$ only, 
$v$ is associated with at most $|A|+1$ nodes of $t_{A}(G'')$ if $v$ belongs to $G''$ only, or
$v$ is associated with at most $|A'|=|A|-1$ nodes of $t_{A'}(G')$ and a single node of $t_A(G'')$ if $v$ belongs to both $G'$ and $G''$.
In the first and the last case,
only one of the $|A'|$ nodes of $t_{A'}(G')$ associated with $v$ has the property that the third coordinate of its color does not contain $\leftarrow$.
In the middle case, exactly one of the $|A|+1$ nodes of $t_{A}(G'')$ associated with $v$ has this property.
\end{proof}

The following lemma is crucial to build the interpretation scheme.

\begin{lemma}
\label{lm-pw-fo}
Fix a finite set $A$.
\begin{itemize}
\item There exists a first order formula $\Phi^A_v(u,u')$ that describes whether two nodes of $t_A(G)$ are associated with the same vertex of a graph $G$.
\item There exists a first order formula $\Phi^A_e(u,u')$ that describes whether two nodes of $t_A(G)$ are associated with adjacent vertices of a graph $G$.
\end{itemize}
Moreover, both formulas $\Phi^A_v(u,u')$ and $\Phi^A_e(u,u')$ are local.
\end{lemma}

\begin{proof}
Let $\Phi(u,v,w)$ be the first order formula that is satisfied
if $u$ is a node colored with $(x,X,Z)$ where $Z$ contains $\rightarrow$,
$v$ is a node colored with $(x,X',Z')$ where $Z'$ contains $\leftarrow$ (note that the first coordinate of the two color triples is the same), and
$w$ is a node on the path from $u$ to the root (possibly $u=w$) such that $v$ is on the path from the successor of $w$ that
always goes to the first child of a node.
Note that such a first order formula exists since the depth of $T$ is bounded and so are the lengths of the paths from $u$ to $w$ and from $w$ to $v$.

Let $\Phi'(u,v)$ be the first order formula that is satisfied
if there exists a vertex $w$ such that $\Phi(u,v,w)$ and
there are no vertices $w'$ and $v'$ such that $\Phi(u,v',w')$ and $w'$ is on the path from $u$ to $w$ (including $u$ but excluding $w$).
Note that if two nodes $u$ and $v$ of $t_A(G)$ satisfy $\Phi'(u,v)$, then they are associated with the same vertex of $G$.
We can now iterate: if three nodes $u$, $v$ and $v'$ satisfy $\Phi'(u,v)$ and $\Phi(v,v')$,
then they are associated with the same vertex of a graph $G$,
if four nodes $u$, $v$, $v'$ and $v''$ satisfy $\Phi'(u,v)$, $\Phi(v,v')$ and $\Phi(v',v'')$,
then they are associated with the same vertex of a graph $G$, etc.
This fully characterizes nodes associated with the same vertex of $G$ (follow the splits in the recursive definition of $t_A$).
Hence, we can define the first order formula $\Phi^A_v(u,u')$ as
$$\exists v_1,\ldots,v_{|A|-1}\;\Phi''(u,v_1)\land\Phi''(v_1,v_2)\land\cdots\land\Phi''(v_{|A|-1},u')\;\mbox{,}$$
where $\Phi''(u,v)$ is the first order formula $(u=v)\lor\Phi'(u,v)\lor\Phi'(v,u)$.
Note that it is enough to consider $|A|-1$ intermediate nodes
since each vertex is associated with at most $|A|+1$ nodes of $t_A(G)$.

The definition of the formula $\Phi^A_e(u,u')$ is now easy.
Two vertices $w$ and $w'$ of $G$ are adjacent
if there exists nodes $v$ and $v'$ associated with them (not necessarily in this order, i.e., $v$ can be associated with $w'$ and $v'$ with $w$) such that
$v$ is in the subtree of $v'$, the color of $v$ is $(x,X,Z)$, the color of $v'$ is $(x',X',Z')$ and $x'\in X$.
So, we set the formula $\Phi^A_e(u,u')$ to express that there exist nodes $v$ and $v'$ such that
either $\Phi^A_v(u,v)\land\Phi^A_v(u',v')$ or $\Phi^A_v(u,v')\land\Phi^A_v(u',v)$,
$v$ is in the subtree of $v'$, the color of $v$ is $(x,X,Z)$, the color of $v'$ is $(x',X',Z')$ and $x'\in X$.

Since the depth of the tree $T$ is bounded by $|A|$,
it is easy to modify the quantifications in the constructed formulas to obtain local formulas.
\end{proof}

We now describe the interpretation scheme $I$ of graphs with bounded path-width in colored plane forests.
Our aim is to invert the mapping $t_{[p+1]}$.
To do so, we have to present a first order formula $\Phi_0(u)$, which determines which nodes of $I$ will be picked to correspond to vertices, and
a first order formula $\Phi_1(u,u')$, which determines the adjacency between the vertices.
The first order formula $\Phi_0(u)$ picks nodes such that the third color of their coordinate does not contain $\leftarrow$;
each vertex is associated with exactly one such node by Lemma~\ref{lm-pw}.
The formula $\Phi_1(u,u')$ is just the formula $\Phi^{[p+1]}_e(u,u')$ from Lemma~\ref{lm-pw-fo}.

The construction of the mappings $t_A$, the formulas $\Phi_0$ and $\Phi_1$, and Lemma~\ref{lm-pw} imply the following.

\begin{lemma}
\label{lm-pw-scheme}
If $G$ is a graph with path-width $p$ and $G'$ is a $[p+1]$-interval graph obtained from $G$
by equipping it with a semi-interval representation and a vertex coloring,
then $I(t_{[p+1]}(G'))$ is a graph isomorphic to $G$.
Moreover, the number of nodes of the tree $t_{[p+1]}(G')$ is at most $(p+1)|G|+1$.
\end{lemma}

We are now ready to prove the remaining of our results.

\begin{theorem}
\label{thm-pw}
For every integer $p$, every first order convergent sequence of graphs with path-width at most $p$ has a limit modeling
which satisfies the strong finitary mass transport principle.
\end{theorem}

\begin{proof}
Let $(G_n)_{n\in\NN}$ be a first order convergent sequence of graphs with path-width at most $p$.
Consider a first order convergent sequence $(T_n)_{n\in\NN}$ of $((p+1)2^{p+3})$-colored plane trees that
is a subsequence of $t_{[p+1]}(G_n)$.
By Theorem~\ref{thm-omega}, this sequence has a limit modeling $M_T$.
Lemmas~\ref{lm-scheme} and~\ref{lm-pw-scheme} yield that $I(M_T)$ is a limit modeling of the sequence $(G_n)_{n\in\NN}$.

It remains to show that the constructed modeling $I(M_T)$ satisfies the strong finitary mass transport principle.
To show this, we need to inspect the construction given in the proof of Lemma~\ref{lm-composition} and the mapping $t_{[p+1]}$.
Lemma~\ref{lm-pw-fo} implies that the edges of $G$ can be described by a first order local formula $\Phi^{[p+1]}_e(u,u')$.
The $d$-Hintikka types of $u$ and $u'$, where $d$ is the quantifier depth of $\Phi^{[p+1]}_e$, and
the relative position of the nodes $u$ and $u'$ in the tree $T$ determine whether the formula $\Phi^{[p+1]}_e(u,u')$ is satisfied.
More precisely, the relative position $u$ and $u'$ plays a role only if one can be reached from the other
by a bounded number of moves, where the bound depends on $\Phi^{[p+1]}_e$ only, formed by traversing the edges in the direction to the root,
traversing only the edges to the first child in the direction from the root, and making the predecessor and successor steps.
The initial and the final part of these moves correspond to ``walking'' from the nodes $u$ and $u'$
to nodes $v$ and $v'$ associated with the same vertices as $u$ and $u'$,
such that $\Phi^{[p+1]}_v(u,v)$ and $\Phi^{[p+1]}_v(u',v')$ where $\Phi^{[p+1]}_v$ is the formula from Lemma~\ref{lm-pw-fo}, and
the middle part consists of one or more moves from to $v$ to $v'$ (or vice versa) following the edges towards the root only.
By Lemma~\ref{lm-pw}, the pair $v$ and $v'$ is unique for every pair $u$ and $u'$ adjacent in $G$ and
thus the (shortest) sequence of the moves from $u$ to $u'$ is always unique.

We have shown that we can split the edge set of $I(M_T)$ into finitely many sets corresponding
to different combinations of pairs of $d$-Hintikka types of $u$ and $u'$ and the relative positions of $u$ and $u'$.
Fix one such combination. Let $F$ be the set of the corresponding edges, and
let $C$ and $D$ be the sets of nodes of these two $d$-Hintikka types.
We will also refer to the nodes in $C$ and $D$ as vertices when we view them as vertices of $I(M_T)$.
Let $f_C$ and $f_D$ be the mappings corresponding to the initial and the final parts of the moves as described in the previous paragraph.
Since each of the moves in these parts corresponds to an invertible measure preserving transformation between the nodes of $M_T$,
both $f_C$ and $f_D$ are invertible measure preserving transformations (with respect to the measure of $M_T$).
By symmetry, we can assume that $f_C(u)$ is in the subtree of $f_D(u')$ for every $uu'\in F$ (otherwise we swap the considered $d$-Hintikka types).

Observe that $uu'$ is an edge if and only if the color of $f_C(u)$ is $(x,X,Z)$, the color of $f_D(u')$ is $(x',X',Z')$ and $x'\in X$.
This implies that the edges of $F$ induce a forest $T_F$ in $I(M_T)$.
Since all the vertices of $I(M_T)$ contained in $C$ have degree one in this forest,
we can define $f_T$ to be the mapping from $C$ to $D$ that assigns $u\in C$ its unique neighbor in the forest $F_T$.
Let $D_i$, $i\in\NN$, be the set of vertices of degree exactly $i$ in the forest $T_F$.
Observe that the nodes of $M_T$ contained in $D_i$ can be characterized by a first order formula and
thus each $D_i$ is measurable. Let $C_i$ be the neighbors of the vertices of $D_i$ in $T_F$ and
let 
$$C_{\infty}=C\setminus\bigcup_{i\in\NN} C_i\mbox{ and }D_{\infty}=D\setminus\bigcup_{i\in\NN} D_i\;\mbox{.}$$
Note that $C_{\infty}$ contains the neighbors of the vertices in $D_{\infty}$, and
all sets $C_i$, $i\in\NN$, $C_{\infty}$ and $D_{\infty}$ are measurable.
We must now dive deeply into the construction presented in the proof of Lemma~\ref{lm-composition}.
For every $u\in C_i$, $i\in\NN$,
all the edges on the path from $f_C(u)$ to $f_D(f_T(u'))$ in $M_T$ are finitary (see the definition in the proof of Lemma~\ref{lm-composition}).
The way the finitary edges are defined allows us to split each set $C_i$ into $i$ measurable sets $C_{i,1},\ldots,C_{i,i}$ such that
$f_T$ is an invertible measure preserving transformation from $C_{i,j}$ to $D_i$ (with respect to the measure of $M_T$).
Note that since the measure of $C$ is finite, the sum of the measures of $C_{i,j}$ is finite.
On the other hand, for every $u\in C_{\infty}$, at least one of the edges on the path $f_C(u)$ to $f_D(f_T(u'))$ in $M_T$ is infinitary.
The way the infinitary edges are introduced using the function $\zeta$ implies that
there is no measurable subset $C'\subseteq C_{\infty}$ with positive measure such that
the maximum degree of the subgraph of $F_T$ induced by $C'$ and $D_{\infty}$ is finite.
In addition, the measure of the set $D_{\infty}$ is zero and
for every measurable subset $D'\subseteq D_{\infty}$,
the subset of $C_{\infty}$ formed by the neighbors of $D'$ is measurable.

We conclude that the edge set of $I(M_T)$ can be split
into countably many sets $F_i$, $i\in\NN$, of edges between measurable sets $U_i$ and $V_i$,
which are the sets $C_{i,j}$ and $D_i$ from the previous paragraph, and
finitely many sets $F'_i$, $i\in [k]$, between measurable sets $C'_i$ and $D'_i$,
which are the sets $C_{\infty}$ and $D_{\infty}$.
Moreover, the sets $F_i$ together with $U_i$ and $V_i$ satisfy that
\begin{itemize}
\item for every $i\in\NN$, there exists an invertible measure preserving transformation $f_i$ from $U_i$ to $V_i$ such that
      the edges of $F_i$ are precisely the edges $uf_i(u)$, $u\in U_i$, and
\item the sum of the measures of $U_i$, $i\in\NN$, is finite.
\end{itemize}
The latter is true since the sets $U_i$ are the sets obtained by splitting finitely many sets $A$, each of measure at most one, into sets $C_{i,j}$.
Note that $f_i$ are invertible measure preserving transformations both in $M_T$ and in $I(M_T)$
since the measure on $I(M_T)$ is a multiple of the measure on $M_T$.
Moreover, every set $F'_i$ together with $U'_i$ and $V'_i$, $i\in [k]$, satisfy that
\begin{itemize}
\item the edges of $F'_i$ form a forest with leaves in $U'_i$ and the central vertices in $V'_i$ for $i\in [k]$,
\item each set $V'_i$, $i\in [k]$, has measure zero,
\item each measurable subset $X$ of $U'_i$ such that
      the maximum degree of the subgraph of $I(M_T)$ induced by the edges from $F'_i$ between $X$ and $V'_i$ is finite
      has zero measure, and
\item if $Y$ is measurable subset of $V'_i$,
      then the subset of $U'_i$ formed by the end-vertices of the edges in $F'_i$ with one end-vertex in $Y$ is also measurable.
\end{itemize}
We are now ready to verify the strong finitary mass transport principle.
Let $A$ and $B$ be two measurable subsets of vertices of $I(M_T)$ such that 
each vertex of $A$ has at least $a$ neighbors in $B$ and each vertex of $B$ has at most $b$ neighbors in $A$.
Our aim is to show that $a\mu(A)\le b\mu(B)$ where $\mu$ is the probability measure associated with $I(M_T)$.
Since removing the vertices of $V'_i$ from $A$ does not decrease the measure of $A$ since the measure of $V'_i$ is zero,
we can assume that $A$ contains no vertex from a set $V'_i$ for every $i\in [k]$.
Fix $i\in [k]$. Let $A_i$ be the set of the end-vertices of the edges in $F'_i$ with one end-vertex in $B\cap V'_i$.
The set $A_i$ is measurable since it is the intersection of the set $A$ and
the subset of $U'_i$ formed by the end-vertices of the edges in $F'_i$ with one end-vertex in $B\cap V'_i$.
Since every vertex of $V'_i$ has at most $b$ neighbors in $A_i$, the measure of $A_i$ is zero.
Hence, we can remove all the vertices of $A_i$ from $A$ without decreasing the measure of $A$.

We have shown that it is possible to assume without loss of generality that
the set $A$ is disjoint from all sets $U'_i$ and $V'_i$, $i\in [k]$.
Hence, all edges between $A$ and $B$ belong to sets $F_i$, $i\in\NN$.
Each vertex of $A$ is contained in at least $a$ sets $f_i(U_i\cap B)$ and $f^{-1}_i(V_i\cap B)$, $i\in\NN$.
Since the maps $f_i$ are invertible measure preserving transformations from $U_i$ to $V_i$ and
the sum of the measures of the sets $U_i$ is finite (and so is the sum of the measures of $V_i$),
we obtain that
\begin{equation}
\sum_{i\in\NN}\mu(A\cap f_i(U_i\cap B))+\mu(A\cap f_i^{-1}(V_i\cap B))\ge a\mu(A)\;\mbox{.}\label{eq-muA}
\end{equation}
A symmetric argument using that each vertex of $B$ is contained in at most $b$ sets $f_i(U_i\cap A)$ and $f^{-1}_i(V_i\cap A)$, $i\in\NN$, yields that
\begin{equation}
\sum_{i\in\NN}\mu(B\cap f_i(U_i\cap A))+\mu(B\cap f_i^{-1}(V_i\cap A))\le b\mu(B)\;\mbox{.}\label{eq-muB}
\end{equation}
Observe that
$$f_i^{-1}(A\cap f_i(U_i\cap B))=f_i^{-1}(A\cap V_i\cap f_i(U_i\cap B))=(U_i\cap B)\cap f_i^{-1}(V_i\cap A)=B\cap f_i^{-1}(V_i\cap A)\;\mbox{.}$$
Since each $f_i$ is measure preserving, the first sum in (\ref{eq-muA}) and the second sum in (\ref{eq-muB}) are equal.
Likewise, the second sum in (\ref{eq-muA}) and the first sum in (\ref{eq-muB}) are equal.
Since the left hand sides of (\ref{eq-muA}) and (\ref{eq-muB}) are equal, it follows that $a\mu(A)\le b\mu(B)$ as desired.
This completes the proof that $I(M_T)$ satisfies the strong finitary mass transport principle.
\end{proof}

\section*{Acknowledgement}

The authors would like to thank Andrzej Grzesik, Anita Liebenau and Fiona Skerman
for discussions on the existence of modelings for first order convergent sequences of sparse graphs, and
Bernard Lidick\'y for his comments on the early write-up of this manuscript.
They would also like to thank two anonymous referees for their very insightful and detailed comments,
including pointing out a need to give additional details on a variant of Hanf's theorem
used in the proof of Lemma~\ref{lm-modeling}; this variant is now stated as Lemma~\ref{lm-hanf}.

\end{document}